\documentclass[a4paper,12pt]{article}

\usepackage{microtype}
\usepackage{cite}
\usepackage[centertags]{amsmath}
\usepackage{amsfonts}
\usepackage{amssymb}
\usepackage{amsthm}
\usepackage{dsfont}
\usepackage{tikz, subfigure}
\usepackage{verbatim}
\usepackage{graphicx}
\usepackage[shortlabels]{enumitem}
\setlist{nosep}
\usepackage{ upgreek }

\usepackage{hyperref}
\hypersetup{
  colorlinks=true,
  linkcolor=black,
  anchorcolor=black,
  citecolor=black,
  filecolor=black,      
  menucolor=red,
  runcolor=black,
  urlcolor=black,
}

\numberwithin{equation}{section}
\newtheorem{theorem}{Theorem}[section]
\newtheorem{lemma}[theorem]{Lemma}

\newtheorem{prop}[theorem]{Proposition}

\newtheorem*{ack}{Acknowledgement}

\newcommand{\HH}{\mathcal{H}}
\newcommand{\LL}{\mathcal{L}}

\newcommand{\MM}{\mathcal{M}}

\newcommand{\CC}{\mathcal{C}}

\newcommand{\VV}{\mathcal{V}}
\newcommand{\R}{\mathbb{R}}
\newcommand{\RP}{\mathbb{RP}^1}

\newcommand{\N}{\mathbb{N}}
\newcommand{\hhh}{\mathtt{h}}
\newcommand{\iii}{\mathtt{i}}
\newcommand{\jjj}{\mathtt{j}}

\newcommand{\vvv}{\mathbf{v}}
\newcommand{\eps}{\varepsilon}
\newcommand{\fii}{\varphi}
\newcommand{\roo}{\varrho}
\newcommand{\ualpha}{\overline{\alpha}}
\newcommand{\lalpha}{\underline{\alpha}}

\newcommand{\A}{\mathsf{A}}

\newcommand{\dd}{\,\mathrm{d}}

\DeclareMathOperator{\dimloc}{dim_{loc}}
\DeclareMathOperator{\udimloc}{\overline{dim}_{loc}}
\DeclareMathOperator{\ldimloc}{\underline{dim}_{loc}}

\DeclareMathOperator{\dimh}{dim_H}

\DeclareMathOperator{\ldimh}{\underline{dim}_H}

\DeclareMathOperator{\diml}{dim_L}
\DeclareMathOperator{\dimaff}{dim_{aff}}
\DeclareMathOperator{\diam}{diam}

\DeclareMathOperator{\dist}{dist}
\DeclareMathOperator{\proj}{proj}
\DeclareMathOperator{\conv}{conv}

\DeclareMathOperator*{\essinf}{ess\,inf}

\renewcommand{\ge}{\geqslant}
\renewcommand{\le}{\leqslant}
\renewcommand{\geq}{\geqslant}
\renewcommand{\leq}{\leqslant}

\begin{document}
\title{Local Dimension Spectrum for Dominated Planar Self-Affine Sets}

\author{Alex Batsis \and Antti K\"{a}enm\"{a}ki \and Tom Kempton}
\newcommand{\addresses}{\footnotesize%
  (A.~Batsis) \textsc{Department of Mathematics, University of Manchester, UK} \par
  (A.~K\"aenm\"aki) \textsc{University of Eastern Finland, Department of Physics and Mathematics, P.O.\ Box 111, FI-80101 Joensuu, Finland} \\
  \textit{Email address}: \texttt{antti@kaenmaki.net} \par
  (T.~Kempton) \textsc{Department of Mathematics, University of Manchester, UK} \\
  \textit{Email address}: \texttt{thomas.kempton@manchester.ac.uk}
}
\date{\today}
\maketitle

\begin{abstract}
  The local dimension spectrum provides a framework for quantifying the fractal properties of a measure, and it is well understood for non-overlapping self-similar measures. In this article, we study the local dimension spectrum for dominated self-affine measures. By analyzing exact dimensionality, we obtain deterministic results that extend the scope of the local dimension spectrum beyond the almost-sure setting.
\end{abstract}

\section{Introduction} \label{sec:introduction}

Let $(A_1,\ldots,A_N) \in GL_2(\R)^N$ be a tuple of contractive invertible $2 \times 2$-matrices. Given a tuple of translation vectors $(v_1,\ldots,v_N) \in (\R^2)^N$, the collection $(\fii_1,\ldots, \fii_N)$ of contractive invertible affine maps defined by
\[
\fii_i(x)= A_ix+v_i
\]
is called an \emph{affine iterated function system (affine IFS)}. Given an affine IFS, there exists a unique non-empty compact set $X \subset \R^2$ such that
\begin{equation} \label{eq:self-affine-set-def}
  X = \bigcup_{i=1}^N \fii_i(X).
\end{equation}
The self-affine set $X$ can be expressed as the image $\pi(\Sigma)$ of the shift space $\Sigma = \{1,\ldots,N\}^\N$, where $\pi \colon \Sigma \to X$ is the canonical projection of the IFS and $\sigma \colon \Sigma \to \Sigma$ is the left shift; see Sections \ref{sec:exact-dimensionality} and \ref{sec:shift-space} for the corresponding definitions. If $\phi \colon \Sigma \to \R$ is a continuous function and $\mu$ is a Borel probability measure ergodic with respect to $\sigma$, then by the Birkhoff ergodic theorem,
\begin{equation*}
  \lim_{n \to \infty} \frac{1}{n} \sum_{j=0}^{n-1} \phi(\sigma^j\iii) = \int \phi \dd\mu
\end{equation*}
for $\mu$-almost all $\iii \in \Sigma$. The Birkhoff ergodic theorem does not provide information about the exceptional set on which this equality fails. Since any two distinct ergodic measures are mutually singular, the exceptional set supports other ergodic measures and therefore cannot be regarded as negligible. B\'ar\'any, Jordan, K\"aenm\"aki, and Rams \cite[Theorem 2.1]{BaranyJordanKaenmakiRams2020} analyzed the structure of this exceptional set for self-affine sets. They determined the Hausdorff dimension of the canonical projection of the set of $\iii \in \Sigma$ for which the Birkhoff averages $\frac{1}{n} \sum_{j=0}^{n-1} \phi(\sigma^j\iii)$ converge to a given value. As a consequence, they were also able to determine the Hausdorff dimension spectrum of the Lyapunov exponents. A much harder problem on self-affine sets is to understand the local dimension spectrum.

Given a probability vector $(p_1,\ldots,p_k)$, we let the \emph{self-affine measure} $\mu$ be the unique probability measure satisfying
\[
\mu=\sum_{i=1}^N p_i\mu\circ \fii_i^{-1}.
\]
The \emph{local dimension} of $\mu$ at $x$ is given by
\begin{equation*}
  \dimloc(\mu,x) = \lim_{r \downarrow 0} \frac{\log\mu(B(x,r))}{\log r}
\end{equation*}
provided the limit exists. If the limit does not exist, the corresponding upper and lower limits are denoted by $\udimloc(\mu,x)$ and $\ldimloc(\mu,x)$, respectively. The local dimension of $\mu$ is intrinsically connected to the dimension of subsets of $X$. It is easy to see that
\begin{equation*}
  \essinf_{x \sim \mu} \ldimloc(\mu,x) = \ldimh(\mu),
\end{equation*}
where $\ldimh(\mu) = \inf\{\dimh(A) : A \subset X \text{ is a Borel set such that }\mu(A)>0\}$ is the \emph{lower Hausdorff dimension} of $\mu$. Let $s \ge 0$ and define the \emph{$s$-level set} of $X$ with respect to $\mu$ to be
\begin{equation*}
  X(\mu,s) = \{x\in X : \dimloc(\mu,x)=s\}.
\end{equation*}
We are interested in determining the Hausdorff dimension of these level sets. 
The \emph{$L^q$-spectrum} of $\mu$ is defined as
\begin{equation*}
  \tau(\mu,q) = \liminf_{r \downarrow 0} \frac{\log\sup\sum_i\mu(B(x_i,r))^q}{\log r},
\end{equation*}
where the supremum is taken over all collections of pairwise disjoint balls $\{B(x_i,r)\}_i$ centered in the support of $\mu$. We say that the \emph{multifractal formalism holds} for $\mu$ at $s$ if
\begin{equation*}
  \dimh(X(\mu,s)) = \inf\{qs-\tau(\mu,q) : q \in \R\}.
\end{equation*}
Our goal is to investigate the \emph{local dimension spectrum} of $\mu$, that is, the behavior of the function $s \mapsto \dimh(X(\mu,s))$, and to provide tools to assess when the multifractal formalism holds for $\mu$. This spectrum is well understood for self-similar measures satisfying the open set condition, with some progress in the overlapping self-similar and self-affine cases. Notable results on self-similar sets and self-affine carpets involving exact overlaps include \cite{King1995,Olsen1998,BarralMensi2007,JordanRams2011,BarralFeng2012}; results for Hueter-Lalley type self-affine sets are given in \cite{FraserKempton2018}.

In particular, for a tuple of matrices $(A_1, \ldots, A_N)$, each with $\|A_i\| < \frac{1}{2}$, and a probability vector $(p_1, \ldots, p_N)$, Barral and Feng \cite{BarralFeng}, building on Falconer’s work \cite{Falconer1999}, derived a formula for part of the corresponding local dimension spectrum, valid for almost every choice of translation vectors $v_1, \ldots, v_N$. Falconer \cite{Falconer1999} established (see \cite[Theorem 1.2]{BarralFeng}) that the symbolic $L^q$-spectrum $\tau$, defined in Section \ref{sec:symbolic-spectra}, coincides with the $L^q$-spectrum for almost all translation vectors when $1 < q \leq 2$. Since $\tau$ is concave, we write $\tau_\pm'(q)$ for its left and right derivatives. We adopt the convention that whenever a one-sided derivative is used at some $q$, the same choice is used throughout.

\begin{theorem}[Barral-Feng \mbox{\cite[Theorem 1.3]{BarralFeng}}] \label{thm:1.1}
Suppose that $(A_1,\ldots, A_N) \in GL_2(\mathbb R)^N$ is such that $\|A_i\|<\frac{1}{2}$ for all $i \in \{1,\ldots,N\}$ and $(p_1,\ldots,p_N)$ is a probability vector. Given a choice of translation vector $\mathbf{v}=(v_1,\ldots,v_N)$, let $\mu_{\mathbf{v}}$ be the corresponding self-affine measure. If $0<q<1$, $0<q\tau_\pm'(q)-\tau(q) < 1$, and $0<\tau(q)/(q-1)<1$, then $\tau(\mu_{\mathbf{v}},q) = \tau(q)$ and
\[
  \dimh(X(\mu_{\mathbf{v}},\tau_\pm'(q))) = q\tau_\pm'(q) - \tau(q)
\]
for Lebesgue almost every choice of $\mathbf{v}$.
\end{theorem}

This almost-sure result relies on a transversality argument, in the spirit of Falconer's work on the almost-sure Hausdorff dimension of self-affine sets; see \cite{FalconerAE}. For $q > 1$, Barral and Feng provide analogous results when the matrices $A_i$ are diagonal. According to \cite[\S 7]{BarralFeng}, the results of Barral and Feng extend to a broader class of quasi-Bernoulli measures, beyond self-affine measures (which are canonical projections of Bernoulli measures).

In recent years there has been substantial progress in understanding the Hausdorff dimension of self-affine sets and measures. In the 1980s Falconer \cite{FalconerAE} gave a formula for the generic dimension of a self-affine set and showed that this formula holds typically. However, there were few explicit examples of self-affine sets satisfying Falconer's formula until the 2010s. More recently, sufficient conditions for the validity of Falconer's formula were given in terms of properties of orthogonal projections of the self-affine measure; see \cite{BaranySA,BaranyKaenmaki2017,KemptonFalconer,Feng2023}. These were then used to show that wide classes of self-affine sets satisfy Falconer's formula; see \cite{BaranyHochmanRapaport,Rapaport2024,MorrisSert2023preprint}. 

Our goal is to mirror recent progress on the Hausdorff dimension, replacing the ``almost-everywhere'' statements in \cite{BarralFeng} with explicit projection results and extending their range of applicability. The following proposition completely characterizes the behavior of the symbolic spectrum. Here $\diml(\eta)$ is the Lyapunov dimension of a shift-invariant measure $\eta$ and $\diml(\mu,\eta)$ is the Lyapunov cross-dimension of $\mu$ relative to $\eta$; see Section \ref{sec:lyapunov-dimension}. A more detailed statement can be found in Theorem \ref{LegendreThm}.

\begin{prop}\label{1.1}
  Let $(A_1,\ldots,A_N) \in GL_2(\R)^N$ be a dominated tuple of contractive matrices, $\mu$ be a fully supported quasi-Bernoulli measure on $\Sigma$, and $q \ne 1$. If $\tau'(q)$, $\tau(q)/(q-1)$, and $q\tau'(q)-\tau(q)$ are contained in the same interval, $(0,1)$, $(1,2)$, or $(2,\infty)$, then $\tau'(q) \le \tau(q)/(q-1)$ and
  \begin{equation*}
    \sup\{\diml(\eta):\eta\text{ is such that }\diml(\mu,\eta)=\tau'(q)\} = q\tau'(q)-\tau(q).
  \end{equation*}
\end{prop}

Note that the interval containing $\tau(q)/(q-1)$ determines which singular values must be considered. The Lyapunov dimension of the auxiliary measure that attains the above supremum and the Lyapunov cross-dimension of $\mu$ relative to $\eta$ are identified via $q\tau'(q)-\tau(q)$ and $\tau'(q)$, respectively, in the same interval as $\tau(q)/(q-1)$.

The Furstenberg direction set $X_F$ is the closure of the top Oseledets directions $V$ for the inverse linear action; see Section \ref{sec:domination} for details. A planar self-affine set $X$ is dominated, for instance, when the linear parts of its defining affine maps have positive entries; see Section \ref{sec:domination} for a precise definition. Throughout this work, we assume that $X$ is dominated and that $X_F$ is not a singleton. This holds, for example, when a dominated self-affine set is irreducible, meaning that the linear parts of its defining affine maps share no common invariant line; see Section \ref{sec:exact-dimensionality}. Finally, we say that $X$ satisfies the \emph{projective strong separation condition} if
\begin{equation} \label{eq:PSSC}
  \proj_{V^\bot}(\conv(\fii_i(X))) \cap \proj_{V^\bot}(\conv(\fii_i(X))) = \emptyset
\end{equation}
for all $V \in X_F$ whenever $i \ne j$, where $\proj_{V^\perp}$ is the orthogonal projection onto $V^\perp$. It is evident that the projective strong separation condition implies the projective open set condition defined in \cite[\S 3]{BaranyKaenmakiYu2021}. It also implies the \emph{strong separation condition}, meaning that $\fii_i(X) \cap \fii_j(X) = \emptyset$ whenever $i \ne j$. Our first main result covers the dimension range $(0,1)$.

\begin{theorem}\label{1.2}
  Suppose that $X$ is a dominated planar self-affine set satisfying the projective strong separation condition such that $X_F$ is not a singleton, and that $\mu$ is a fully supported quasi-Bernoulli measure on $X$. If $0<q\ne 1$, $0<q\tau'(q)-\tau(q)<1$, and $0<\tau(q)/(q-1)<1$, then
  \begin{equation*}
    \dimh(X(\mu,\tau'(q))) = q\tau'(q)-\tau(q).
  \end{equation*}
\end{theorem}

To gain further insight into the projective strong separation condition, observe that $\dimh(X) < 1$ since $X_F$ is invariant. Hence, this condition implies all equivalent statements in \cite[Theorem 3.5]{BaranyKaenmakiYu2021}. In particular, it ensures that $\proj_{V^\bot}(X)$ is Ahlfors regular for all $V \in X_F$; thus, measures supported on $\proj_{V^\bot}(X)$ cannot be absolutely continuous. Examples of dominated, irreducible planar self-affine sets satisfying the projective strong separation condition are given in \cite[Corollary 1.4]{BaranyKaenmakiYu2021}. The second main result covers the dimension range $(1,2)$.

\begin{theorem}\label{1.3}
  Suppose that $X$ is a dominated planar self-affine set satisfying the strong separation condition such that $X_F$ is not a singleton, and that $\mu$ is a fully supported quasi-Bernoulli measure on $X$ such that $(\proj_{V^\bot})_*\mu$ is absolutely continuous with uniformly bounded density over all $V \in X_F$. If $0<q\ne 1$, $1<q\tau'(q)-\tau(q)<2$, and $1<\tau(q)/(q-1)<2$, then
  \begin{equation*}
    \dimh(X(\mu,\tau'(q))) = q\tau'(q)-\tau(q).
  \end{equation*}
\end{theorem}

Theorems \ref{1.2} and \ref{1.3} are proved via a sequence of propositions stated in Section \ref{sec:multifractal-formalism}. The propositions in Section \ref{sec:multifractal-formalism} are slightly more general than Theorems \ref{1.2} and \ref{1.3}; for example, the proof of the lower bound in Theorem \ref{1.3} requires only that the projections of $\mu$ have dimension one, not that they are absolutely continuous with bounded density.

Recall that, by \cite[Lemma 2.10]{BaranyKaenmakiYu2021} and \cite[Theorem 7.1]{BaranyHochmanRapaport}, we have 
\begin{equation*}
  \ldimh((\proj_{V^\perp})_*\mu) = 1
\end{equation*}
for all $V \in \RP$ and fully supported self-affine measures $\mu$ on $X$ with $\ldimh(\mu) \geq 1$. Recently, B\'ar\'any \cite[Theorem 1.2(i)]{Barany2025} strengthened this result for the equilibrium state of the singular value function. Specifically, if $\HH^s(X) > 0$, where $1 < s \leq 2$ is the affinity dimension, then the equilibrium state satisfies the conditions of Theorem \ref{1.3}. The affinity dimension and the equilibrium state for the singular value function are defined in Section \ref{sec:equilibrium-states}. The following proposition extends this by adapting the techniques of \cite{Barany2025} to construct a class of self-affine measures satisfying the conditions of Theorem \ref{1.3}.

\begin{prop} \label{thm:abs-cont}
  Let $\fii_i \colon \R^2 \to \R^2$, $\fii_i(x) = A_ix+v_i$, where $v_i = (v_i^1,v_i^2)$ and
  \begin{equation*}
    A_i =
    \begin{pmatrix}
      a_i & 0 \\ 
      b_i & c_i
    \end{pmatrix}
  \end{equation*}
  with $0<|a_i|<|c_i|<\tfrac12$ for all $i \in \{1,\ldots,N\}$ such that the matrices $A_i$ are not simultaneously diagonalizable. If the associated self-affine set $X$ satisfies the strong separation condition and $\mu$ is a fully supported quasi-Bernoulli measure satisfying
  \begin{equation} \label{eq:abs-cont-qb}
    \inf\biggl\{ s \geq 0 : \sum_{n \in \N} \sum_{\mathtt{i} \in \Sigma_n} \frac{ \mu([\mathtt{i}])^2 }{ \|A_\mathtt{i}\|^s } < \infty \biggr\} > 3,
  \end{equation}
  then $\A = (A_1,\ldots,A_N) \in GL_2(\R)$ is dominated, $X_F$ is not a singleton, and $(\proj_{V^\bot})_*\mu$ is absolutely continuous with uniformly bounded density over all $V \in X_F$ for Lebesgue almost every choice of $(v_1^2,\ldots,v_N^2)$.
\end{prop}

We prove Proposition \ref{thm:abs-cont} at the end of Section \ref{sec:multifractal-formalism}. This proposition confirms that the class of self-affine sets and measures satisfying the conditions of Theorem \ref{1.3} is non-empty. We also remark that verifying condition \eqref{eq:abs-cont-qb} for self-affine measures is straightforward; see \eqref{eq:abs-cont-mu3}. Given that this result relies on a random choice of translation vectors, constructing explicit examples remains a challenging and open problem.

The paper is organized as follows. In Sections \ref{sec:thermodynamic-formalism} and \ref{sec:symbolic-spectra}, we work in the shift space, cover preliminaries on thermodynamic formalism, and determine the properties of the symbolic $L^q$-spectrum. In Sections \ref{sec:exact-dimensionality} and \ref{sec:multifractal-formalism}, we work in the plane, study the exact dimensionality of quasi-Bernoulli measures in detail, and transfer the symbolic results to self-affine sets.

\section{Thermodynamic Formalism} \label{sec:thermodynamic-formalism}

\subsection{Shift Space} \label{sec:shift-space}

Let $\Sigma = \{ 1,\ldots,N \}^\N$ be the collection of all infinite words obtained from the alphabet $\{ 1,\ldots,N \}$. If $\iii = i_1i_2\cdots \in \Sigma$, then we define $\iii|_n = i_1 \cdots i_n$ for all $n \in \N$. If $\iii = i_1 \cdots i_n$, then we write $\overleftarrow{\iii} = i_n \cdots i_1$. The empty word $\iii|_0$ is denoted by $\varnothing$. Define $\Sigma_n = \{ \iii|_n : \iii \in \Sigma \}$ for all $n \in \N$ and $\Sigma_* = \bigcup_{n \in \N} \Sigma_n \cup \{ \varnothing \}$. Thus $\Sigma_*$ is the collection of all finite words. The length of $\iii \in \Sigma_* \cup \Sigma$ is denoted by $|\iii|$. The longest common prefix of $\iii,\jjj \in \Sigma_* \cup \Sigma$ is denoted by $\iii \wedge \jjj$ and the concatenation of two words $\iii \in \Sigma_*$ and $\jjj \in \Sigma_* \cup \Sigma$ is denoted by $\iii\jjj$. Let $\sigma$ be the \emph{left shift} operator defined by $\sigma\iii = i_2i_3\cdots$ for all $\iii = i_1i_2\cdots \in \Sigma$. If $\iii \in \Sigma_n$ for some $n$, then we set $[\iii] = \{ \jjj \in \Sigma : \jjj|_n = \iii \}$. The set $[\iii]$ is called a \emph{cylinder set}.

If $\mu \in \MM(\Sigma)$, where $\MM(\Sigma)$ denotes the collection of all Borel probability measures on $\Sigma$, then for a measurable map $f$ defined on $\Sigma$ we denote the pushforward of a measure $\nu$ by $f_*\nu$. We say that a measure $\mu \in \MM(\Sigma)$ is \emph{fully supported} if every cylinder set has positive measure.

\subsection{Lyapunov Dimension} \label{sec:lyapunov-dimension}

We shall consider maps $\theta \colon \Sigma_* \to (0,\infty)$ which we refer to as \emph{potentials}. We say that a potential $\theta$ is \emph{sub-multiplicative} if $\theta(\iii\jjj) \le \theta(\iii)\theta(\jjj)$ for all $\iii,\jjj \in \Sigma_*$. A potential $\theta$ is \emph{super-multiplicative} if the inverse $1/\theta$ is sub-multiplicative. We furthermore say that a potential $\theta$ is \emph{almost-multiplicative} if there is a constant $C \ge 1$ such that $C\theta$ is sub-multiplicative and $C^{-1}\theta$ is super-multiplicative, and \emph{multiplicative} if the constant $C$ can be chosen to be $1$. Let $\MM_\sigma(\Sigma)$ denote the collection of all $\sigma$-invariant Borel probability measures on $\Sigma$. For a sub-multiplicative potential $\theta$ and $\nu \in \MM_\sigma(\Sigma)$, we define
\begin{equation*}
  \Lambda(\theta,\nu) = \lim_{n\to\infty}\frac{1}{n}\int_\Sigma \log\theta(\iii|_n)\dd\nu(\iii).
\end{equation*}
The following well-known lemma guarantees that $\Lambda$ is well-defined.

\begin{lemma} \label{thm:fekete}
  If $\theta$ is a sub-multiplicative potential and $\nu \in \MM_\sigma(\Sigma)$, then $\Lambda(\theta,\nu)$ exists and
  \begin{equation*}
    \Lambda(\theta,\nu) = \inf_{n\in\N}\frac{1}{n}\int_\Sigma \log\theta(\iii|_n)\dd\nu(\iii).
  \end{equation*}
  If $\nu$ is ergodic, then
  \begin{equation*}
    \Lambda(\theta,\nu) = \lim_{n\to\infty}\frac{1}{n} \log\theta(\iii|_n)
  \end{equation*}
  for $\nu$-almost all $\iii \in \Sigma$. Furthermore, $\nu \mapsto \Lambda(\theta,\nu)$ defined on $\MM_\sigma(\Sigma)$ is upper semi-continuous in the weak$^*$ topology.
\end{lemma}

\begin{proof}
  The sequence $(\int_\Sigma \log\theta(\iii|_n)\dd\nu(\iii))_{n \in \N}$ is sub-additive and therefore, by Fekete's Lemma, $\Lambda(\theta,\nu)$ exists and is equal to
  \begin{equation*}
    \inf_{n\in\N}\frac{1}{n}\int_\Sigma \log\theta(\iii|_n)\dd\nu(\iii).
  \end{equation*}
  The second claim follows from Kingman's ergodic theorem; see e.g.\ \cite[Theorem 10.1]{Walters1982}. The third claim is a direct consequence of the first claim as each $\nu \mapsto \frac{1}{n}\sum_{\iii \in \Sigma_n} \nu([\iii])\log\theta(\iii)$ is continuous.
\end{proof}

Let $(A_1,\ldots,A_N) \in GL_2(\R)^N$ and write $A_\iii = A_{i_1} \cdots A_{i_n}$ for all $\iii = i_1 \cdots i_n \in \Sigma_n$ and $n \in \N$. By $A_\iii^{-1}$ we mean $(A_\iii)^{-1}$. For $\iii\in\Sigma_*$ we define $\alpha_1(\iii)$ and $\alpha_2(\iii)$ to be the lengths of the major and minor semi-axis of the ellipse $A_\iii(D)$ respectively, where $D \subset \R^2$ is the unit disc. Note that $\alpha_1(\iii) = \|A_\iii\|$ and $\alpha_2(\iii) = \|A_\iii^{-1}\|^{-1}$ for all $\iii \in \Sigma_*$, where $\|\cdot\|$ is the Euclidean operator norm. The potential $\iii \mapsto \alpha_1(\iii)$ is thus sub-multiplicative and $\iii \mapsto \alpha_2(\iii)$ is super-multiplicative. We define the \emph{Lyapunov exponents} of $\nu \in \MM_\sigma(\Sigma)$ by
\begin{align*}
  \lambda_1(\nu) &= \Lambda(\alpha_1,\nu) = \inf_{n\in\N} \frac{1}{n} \int_\Sigma \log \alpha_1(\iii|_{n}) \dd\nu(\iii), \\
  \lambda_2(\nu) &= -\Lambda(1/\alpha_2,\nu) = \sup_{n\in\N} \frac{1}{n}\int_\Sigma \log \alpha_2(\iii|_{n}) \dd\nu(\iii).
\end{align*}
Note that $\lambda_2(\nu) \le \lambda_1(\nu) \le 0$. Recall that the \emph{entropy} of $\nu \in \MM_\sigma(\Sigma)$ is
\begin{equation*}
  h(\nu) = -\lim_{n \to \infty} \frac{1}{n} \sum_{\iii \in \Sigma_n} \nu([\iii]) \log\nu([\iii]) = \inf_{n \in \N} -\frac{1}{n} \sum_{\iii \in \Sigma_n} \nu([\iii]) \log\nu([\iii]).
\end{equation*}
We say that a fully supported measure $\mu \in \MM(\Sigma)$ is \emph{sub-multiplicative} if the potential $\iii \mapsto \mu([\iii])$ is sub-multiplicative. The other definitions on potentials can be used with measures in a similar manner. Almost-multiplicative measures are more commonly known as \emph{quasi-Bernoulli} measures, while multiplicative measures are called \emph{Bernoulli} measures. The \emph{cross-entropy} of a sub-multiplicative measure $\mu$ relative to $\nu \in \MM_\sigma(\Sigma)$ is defined to be
\begin{equation*}
  h(\mu,\nu) = -\Lambda(\mu,\nu) = \sup_{n \in \N} -\frac{1}{n} \sum_{\iii \in \Sigma_n} \nu([\iii]) \log\mu([\iii]).
\end{equation*}

\begin{lemma} \label{thm:entropy-crossentropy}
  If $\nu \in \MM_\sigma(\Sigma)$ and $\mu \in \MM(\Sigma)$ is sub-multiplicative, then
  \begin{equation*}
    0 \le h(\nu) \le h(\mu,\nu).
  \end{equation*}
\end{lemma}

\begin{proof}
  Since $-x\log x \ge 0$ for all $x \in (0,1]$, we have $h(\nu) \ge 0$. To show the other inequality, fix $n \in \N$ and, as $\log x \le x-1$ for all $x>0$, note that
  \begin{align*}
    -\sum_{\iii \in \Sigma_n} \nu([\iii]) \log\mu([\iii]) &= -\sum_{\iii \in \Sigma_n} \nu([\iii]) \log\frac{\mu([\iii])}{\nu([\iii])} - \sum_{\iii \in \Sigma_n} \nu([\iii]) \log\nu([\iii]) \\ 
    &\ge \sum_{\iii \in \Sigma_n} \nu([\iii]) \biggl(1-\frac{\mu([\iii])}{\nu([\iii])}\biggr) - \sum_{\iii \in \Sigma_n} \nu([\iii]) \log\nu([\iii]) \\ 
    &= -\sum_{\iii \in \Sigma_n} \nu([\iii]) \log\nu([\iii]).
  \end{align*}
  The claim follows by taking supremum and infimum.
\end{proof}

The \emph{Lyapunov dimension} of a measure $\nu \in \MM_\sigma(\Sigma)$ is given by
\begin{equation*}
  \diml(\nu) = \min\biggl\{ -\frac{h(\nu)}{\lambda_1(\nu)}, 1-\frac{h(\nu)+\lambda_1(\nu)}{\lambda_2(\nu)}, -\frac{2h(\nu)}{\lambda_1(\nu)+\lambda_2(\nu)} \biggr\}.
\end{equation*}
Finally, the \emph{Lyapunov cross-dimension} $\diml(\mu,\nu)$ of a sub-multiplicative measure $\mu$ relative to $\nu \in \MM_\sigma(\Sigma)$ is
\begin{equation*}
  \diml(\mu,\nu) = \min\biggl\{ -\frac{h(\mu,\nu)}{\lambda_1(\nu)}, 1-\frac{h(\mu,\nu)+\lambda_1(\nu)}{\lambda_2(\nu)}, -\frac{2h(\mu,\nu)}{\lambda_1(\nu)+\lambda_2(\nu)} \biggr\}.
\end{equation*}
In other words, the Lyapunov cross-dimension is obtained by replacing the entropy $h(\nu)$ in the definition of the Lyapunov dimension by the cross-entropy $h(\mu,\nu)$. Therefore, by Lemma \ref{thm:entropy-crossentropy}, we always have $\diml(\nu) \le \diml(\mu,\nu)$.

\subsection{Domination} \label{sec:domination}

We say that $\mathsf{A} = (A_1,\ldots,A_N) \in GL_2(\R)^N$ is \emph{dominated} if there exist constants $C>0$ and $0<\tau<1$ such that
\begin{equation*}
  \alpha_2(\iii) \le C\tau^n\alpha_1(\iii)
\end{equation*}
for all $\iii \in \Sigma_n$ and $n \in \N$. Note that if $\A$ is dominated, then $\lambda_2(\nu) < \lambda_1(\nu)$ for all $\nu \in \MM_\sigma(\Sigma)$. Let $\RP$ denote the real projective line, which is the set of straight lines going through the origin in $\R^2$ and which we identify with $[0,\pi)$. We call a proper subset $\CC \subset \RP$ a \emph{multicone} if it is a finite union of closed projective intervals. We say that a multicone $\CC \subset \RP$ is \emph{strongly invariant} for $\A$ if $A_i\CC \subset \CC^o$ for all $i \in \{1,\ldots,N\}$, where $\CC^o$ is the interior of $\CC$. For example, the union of the first and third quadrants is strongly invariant for any tuple of positive matrices. Notice that if a multicone $\CC \subset \RP$ is strongly invariant for $\A = (A_1,\ldots,A_N)$, then the closure of $\RP \setminus \CC$ is a strongly invariant multicone for $(A_1^{-1}\ldots,A_N^{-1})$. By \cite[Theorem B]{BochiGourmelon2009}, $\A$ has strongly invariant multicone if and only if $\A$ is dominated. Therefore, if $\A$ is dominated, then the tuple $(A_1^{-1}\ldots,A_N^{-1})$ of inverses has a strongly invariant multicone. Furthermore, if $\A$ is dominated, then \cite[Lemma 2.2]{BochiMorris2015} (see also \cite[Corollary 2.4]{BaranyKaenmakiMorris2018}) implies that the potential $\iii \mapsto \alpha_1(\iii)$ is almost-multiplicative. Since $|\det(A_\iii)| = \alpha_1(\iii)\alpha_2(\iii)$ for all $\iii \in \Sigma_*$ and the determinant is multiplicative, it follows that $\iii \mapsto \alpha_2(\iii) = \alpha_1(\iii)^{-1}|\det(A_\iii)|$ is also almost-multiplicative.

\begin{lemma} \label{thm:almost-cont}
  Let $\theta$ be an almost-multiplicative potential and $\nu \in \MM_\sigma(\Sigma)$. If $\nu_k \to \nu$ in the weak$^*$ topology, then
  \begin{equation*}
    \lim_{k\to\infty} \Lambda(\theta,\nu_k) = \Lambda(\theta,\nu).
  \end{equation*}
\end{lemma}

\begin{proof}
  By Lemma \ref{thm:fekete}, we have $\limsup_{k \to \infty} \Lambda(\theta,\nu_k) \le \Lambda(\theta,\nu)$. Since $C/\theta$ is sub-multiplicative for some $C \ge 1$, Lemma \ref{thm:fekete} implies that
  \begin{equation*}
    \Lambda(C^{-1}\theta,\nu) = -\Lambda(C/\theta,\nu) = \sup_{n \in \N} \biggl( \frac{1}{n} \int_\Sigma \log\theta(\iii|_n) \dd\nu(\iii) - \frac{1}{n} \log C \biggr).
  \end{equation*}
  Therefore, as each $\nu \mapsto \frac{1}{n}\sum_{\iii \in \Sigma_n} \nu([\iii])\log\theta(\iii)$ is continuous, we see that $\nu \mapsto \Lambda(C^{-1}\theta,\nu) = \Lambda(\theta,\nu)$ is lower semi-continuous and thus $\liminf_{k \to \infty} \Lambda(\theta,\nu_k) \ge \Lambda(\theta,\nu)$.
\end{proof}

Let $\A = (A_1,\ldots,A_N) \in GL_2(\R)^N$ be dominated and $\CC \subset \RP$ be a strongly invariant multicone for $\A$. The \emph{canonical projection} $\Pi \colon \Sigma \to \RP$ is defined by the relation
\begin{equation*}
  \{\Pi(\iii)\} = \bigcap_{n=1}^\infty A_{\overleftarrow{\iii|_n}}^{-1} \overline{\RP \setminus \CC}
\end{equation*}
and the set of \emph{Furstenberg directions} is the compact set
\begin{equation*}
  X_F = \bigcup_{\iii \in \Sigma} \Pi(\iii) = \bigcap_{n=1}^\infty \bigcup_{\iii \in \Sigma_n} A_{\overleftarrow{\iii}}^{-1} \overline{\RP \setminus \CC} \subset \overline{\RP \setminus \CC}.
\end{equation*}
Note that $X_F$ is uniformly perfect and hence has positive Hausdorff dimension unless it is a singleton. See \cite[Lemma 2.3]{AnttilaBaranyKaenmaki2024} for alternative characterizations of $X_F$. If $\nu \in \MM_\sigma(\Sigma)$ is ergodic, then the measure $\nu_F = \Pi_*\nu$ on $X_F$ is called the \emph{Furstenberg measure} with respect to $\nu$. Let $T \colon \Sigma \times \RP \to \Sigma \times \RP$ be the skew-product defined by
\begin{equation*}
  T(\iii,V) = (\sigma\iii,A_{\iii|_1}^{-1} V).
\end{equation*}
If $\nu$ is Bernoulli, the measure $\nu \times \nu_F$ on $\Sigma \times \RP$ is $T$-invariant and ergodic. If $\nu$ is only quasi-Bernoulli, then $\nu\times\nu_F$ may not be invariant, but is equivalent to a $T$-invariant ergodic measure; see \cite[Theorem 2.2]{BaranyKaenmaki2017} and \cite[Lemma 3.1]{KemptonFalconer}, respectively.

\subsection{Equilibrium States} \label{sec:equilibrium-states}

Let $\theta$ be a sub-multiplicative potential. We define the \emph{pressure} of $\theta$ by setting
\begin{equation*}
  P(\theta) = \lim_{n\to\infty} \frac{1}{n}\log\sum_{\iii\in\Sigma_n}\theta(\iii) = \inf_{n\in\N} \frac{1}{n}\log\sum_{\iii\in\Sigma_n}\theta(\iii).
\end{equation*}
As in Lemma \ref{thm:fekete}, the existence of the limit above and the equality are guaranteed by Fekete's Lemma. By \cite[Lemma 2.2]{KaenmakiReeve2014}, we see that
\begin{equation*}
  P(\theta) \ge h(\nu) + \Lambda(\theta,\nu)
\end{equation*}
for all $\nu \in \MM_\sigma(\Sigma)$. A measure $\nu \in \MM_\sigma(\Sigma)$ for which
\begin{equation*}
  P(\theta) = h(\nu) + \Lambda(\theta,\nu)
\end{equation*}
is called the \emph{equilibrium state} for $\theta$. If $\theta$ is an almost-multiplicative potential, then, by \cite[Theorem 5]{Mummert2006} (or \cite[Theorem 14]{Barreira2006}), there exists a unique equilibrium state for $\theta$ which furthermore is a fully supported quasi-Bernoulli measure.

\begin{lemma} \label{thm:eq-convergence}
 Let $(\theta_k)_{k\in\N}$ be a sequence of sub-multiplicative potentials and let $\nu_k \in \MM_\sigma(\Sigma)$ be an equilibrium state for $\theta_k$ for each $k \in \N$. If there exist a measure $\nu \in \MM_\sigma(\Sigma)$ and a sub-multiplicative potential $\theta$ such that $\nu_k \to \nu$ in the weak$^*$ topology and $\theta_k(\iii)^{1/|\iii|} \to \theta(\iii)^{1/|\iii|}$ uniformly in $\Sigma_*$ as $k \to \infty$, then
  \begin{equation*}
    \lim_{k \to \infty} \Lambda(\theta_k,\nu_k) - \Lambda(\theta,\nu_k) = 0
  \end{equation*}
  and $\nu$ is an equilibrium state for $\theta$.
\end{lemma}
\begin{proof}
  Recall that, by \cite[Theorem 8.2]{Walters1982} and Lemma \ref{thm:fekete}, $\limsup_{k \to \infty} h(\nu_k) \le h(\nu)$ and $\limsup_{k \to \infty} \Lambda(\theta,\nu_k) \le \Lambda(\theta,\nu)$. If $\eps>0$, then the uniform convergence of $\theta_k$ implies that there exists $k_0 \in \N$ such that $\log\theta(\iii) - \eps|\iii| \le \log\theta_k(\iii) \le \log\theta(\iii) + \eps|\iii|$ for all $\iii \in \Sigma_*$ and
  \begin{equation*}
    \Lambda(\theta,\nu_k) - \eps \le \Lambda(\theta_k,\nu_k) \le \Lambda(\theta,\nu_k) + \eps
  \end{equation*}
  for all $k \ge k_0$. Therefore, $\lim_{k \to \infty} \Lambda(\theta_k,\nu_k) - \Lambda(\theta,\nu_k) = 0$ and
  \begin{align*}
    P(\theta) &= \lim_{k \to \infty} P(\theta_k) = \lim_{k \to \infty} h(\nu_k) + \Lambda(\theta_k,\nu_k) \\ 
    &\le \limsup_{k \to \infty} h(\nu_k) + \limsup_{k \to \infty} \Lambda(\theta,\nu_k) + \eps \\ 
    &\le h(\nu) + \Lambda(\theta,\nu) + \eps.
  \end{align*}
  By letting $\eps \downarrow 0$, we see that $\nu$ is an equilibrium state for $\theta$.
\end{proof}

Let $\mathsf{A} = (A_1,\ldots,A_N) \in GL_2(\R)^N$ be a tuple of contractive invertible matrices. For each $s \ge 0$, define a potential $\fii^s$ by setting
\begin{equation*}
  \fii^s(\iii) =
  \begin{cases}
    \alpha_1(\iii)^s, &\text{if } 0 \le s < 1, \\
    \alpha_1(\iii)\alpha_2(\iii)^{s-1}, &\text{if } 1 \le s < 2, \\ 
    |\det(A_\iii)|^{s/2}, &\text{if } 2 \le s < \infty,
  \end{cases}
\end{equation*}
for all $\iii \in \Sigma_*$. Since $\alpha(\iii)\alpha_2(\iii)^{s-1} = \alpha_1(\iii)^{2-s}|\det(A_\iii)|^{s-1}$, the \emph{singular value function} $\fii^s$ is sub-multiplicative. Therefore, the pressure $P(\fii^s)$ is well-defined for all $s \ge 0$. By \cite[Lemma 2.1]{KaenmakiVilppolainen2010}, the function $s \mapsto P(\fii^s)$ defined on $[0,\infty)$ is continuous and strictly decreasing, convex on intervals $[0,1]$, $[1,2]$, and $[2,\infty)$, and there exists a unique $s \ge 0$ such that $P(\fii^s)=0$. This unique value is called the \emph{affinity dimension}, denoted by $\dimaff(\fii^s)$.

If $\nu \in \MM_\sigma(\Sigma)$, then
\begin{equation*}
  \Lambda(\fii^s,\nu) =
  \begin{cases}
    s\lambda_1(\nu), &\text{if } 0 \le s < 1, \\
    \lambda_1(\nu)+(s-1)\lambda_2(\nu), &\text{if } 1 \le s < 2, \\ 
    \tfrac{s}{2}(\lambda_1(\nu)+\lambda_2(\nu)), &\text{if } 2 \le s < \infty,
  \end{cases}
\end{equation*}
where $\lambda_1(\nu)$ and $\lambda_2(\nu)$ are the Lyapunov exponents. It is straightforward to see that the Lyapunov dimension $\diml(\nu)$ is the unique $s \ge 0$ for which $h(\nu) + \Lambda(\fii^s,\nu) = 0$ or, equivalently,
\begin{equation} \label{eq:diml-min}
  s =
  \begin{cases}
    -\frac{h(\nu)}{\lambda_1(\nu)}, &\text{if } 0 \le s < 1, \\ 
    1-\frac{h(\nu)+\lambda_1(\nu)}{\lambda_2(\nu)}, &\text{if } 1 \le s < 2, \\ 
    -\frac{2h(\nu)}{\lambda_1(\nu)+\lambda_2(\nu)}, &\text{if } 2 \le s < \infty.
  \end{cases}
\end{equation}
Similarly, the Lyapunov cross-dimension $\diml(\mu,\nu)$ of a sub-multiplicative measure $\mu$ relative to $\nu \in \MM_\sigma(\Sigma)$ is the unique $s \ge 0$ for which $h(\mu,\nu) + \Lambda(\fii^s,\nu) = 0$ or, equivalently, $s$ is as in \eqref{eq:diml-min} with $h(\nu)$ replaced by $h(\mu,\nu)$.

By \cite[Theorems 2.6 and 4.1]{Kaenmaki2004}, there exists an ergodic equilibrium state $\nu$ for $\fii^s$. If $\A$ is dominated, then $\fii^s$ is almost-multiplicative and $\nu$ is quasi-Bernoulli. It is easy to see that an equilibrium state $\nu$ has maximal possible Lyapunov dimension, $\diml(\nu) = \max \{\diml(\eta) : \eta\in\MM_\sigma(\Sigma)\} = \dimaff(\fii^s)$.

\section{Differentiability of the Pressure} \label{sec:differentiability-of-the-pressure}

Let $(A_1,\ldots,A_N) \in GL_2(\R)^N$ be dominated and $\mu \in \MM(\Sigma)$ be a fully supported quasi-Bernoulli measure. For each $q \in \R$ and $s \ge 0$, following \cite{Falconer1999}, we consider the almost-multiplicative potential $\psi^{q,s}$ defined by
\begin{equation*}
  \psi^{q,s}(\iii) = \mu([\iii])^q \fii^s(\iii)^{1-q}.
\end{equation*}
If $\nu \in \MM_\sigma(\Sigma)$, then
\begin{equation*}
  \Lambda(\psi^{q,s},\nu) = -qh(\mu,\nu) + (1-q)\Lambda(\fii^s,\nu).
\end{equation*}
Since $\psi^{q,s}$ is almost-multiplicative, the pressure
\[
  P(\psi^{q,s})=\lim_{n\to\infty}\frac{1}{n}\log\sum_{\iii\in\Sigma_n}\mu([\iii])^q\varphi^s(\iii)^{1-q}
\]
is well-defined and there exists a unique ergodic equilibrium state for $\psi^{q,s}$, which is moreover a fully supported quasi-Bernoulli measure. 

In this section, we prove that the function $(q,s) \mapsto P(\psi^{q,s})$ is differentiable on $\R \times (0,\infty) \setminus \{1,2\}$ and compute its partial derivatives. The differentiability of the pressure with respect to several real parameters in the diagonal case for almost-multiplicative potentials is addressed in \cite{BarralFeng2012}. We begin by collecting some elementary properties of the pressure function.

\begin{lemma} \label{thm:pressure-convex}
  If $(A_1,\ldots,A_N) \in GL_2(\R)^N$ is a dominated tuple of contractive matrices and $\mu \in \MM(\Sigma)$ is a fully supported quasi-Bernoulli measure, then the following seven properties hold:
  \begin{enumerate}[(1)]
    \item The function $(q,s) \mapsto P(\psi^{q,s})$ is continuous on $\R \times [0,\infty)$.
    \item For each $q<1$, the function $s \mapsto P(\psi^{q,s})$ is strictly decreasing with $P(\psi^{q,0}) \ge 0$ and $\lim_{s\to\infty}P(\psi^{q,s})=-\infty$.
    \item For each $q>1$, the function $s \mapsto P(\psi^{q,s})$ is strictly increasing with $P(\psi^{q,0}) \le 0$ and $\lim_{s\to\infty}P(\psi^{q,s})=\infty$.
    \item For each $q \ne 1$, there exists a unique $s_q \in [0,\infty)$ so that $P(\psi^{q,s_q})=0$.
    \item The function $q \mapsto s_q$ is continuous on $\R \setminus \{1\}$.
    \item For each $q \in \R$, the function $s \mapsto P(\psi^{q,s})$ is convex on connected components of $[0,\infty) \setminus \{1,2\}$.
    \item For each $s \in [0,\infty) \setminus \{1,2\}$, the function $q \mapsto P(\psi^{q,s})$ is convex on $\R$.
  \end{enumerate}
\end{lemma}

\begin{proof}
  Although the proof is a straightforward modification of \cite[Lemma 2.1]{KaenmakiVilppolainen2010}, we include the details for the convenience of the reader. We prove the claims only for $s \in [0,2)$; the case $s \ge 2$ is left to the reader. Let $p,q \in \R$ and $s,t \in [0,2)$. Write
  \begin{equation*}
    \lalpha = \min_{i \in \{1,\ldots,N\}} \alpha_2(i), \quad
    \ualpha = \max_{i \in \{1,\ldots,N\}} \alpha_1(i), \quad
    K = \max_{i \in \{1,\ldots,N\}} C\mu([i])^{-1},
  \end{equation*}
  where $C \ge 1$ is the constant from the quasi-Bernoulli assumption. Then $0 < \lalpha \le \ualpha < 1 < K$. Furthermore, let
  \begin{equation*}
    \ualpha(q,s,t) =
    \begin{cases}
      \ualpha, &\text{if } (1-q)(s-t) \ge 0, \\ 
      \lalpha, &\text{if } (1-q)(s-t) < 0,
    \end{cases}
  \end{equation*}
  and
  \begin{equation*}
    \lalpha(q,s,t) =
    \begin{cases}
      \lalpha, &\text{if } (1-q)(s-t) \ge 0, \\ 
      \ualpha, &\text{if } (1-q)(s-t) < 0.
    \end{cases}
  \end{equation*}
  Then we have $K^{-|\iii|} \le \mu([\iii]) \le K^{|\iii|}$ and
  \begin{equation*}
    \fii^t(\iii)^{1-q} \lalpha(q,s,t)^{(s-t)(1-q)|\iii|} \le \fii^s(\iii)^{1-q} \le \fii^t(\iii)^{1-q} \ualpha(q,s,t)^{(s-t)(1-q)|\iii|}
  \end{equation*}
  for all $\iii \in \Sigma_*$. Since $\lalpha^{t|\iii|} \le \fii^t(\iii) \le \ualpha^{t|\iii|} \le \lalpha^{-t|\iii|}$, we have
  \begin{equation*}
    \fii^t(\iii)^{1-p} \lalpha^{t|p-q||\iii|} \le \fii^t(\iii)^{1-q} \le \fii^t(\iii)^{1-p} \lalpha^{-t|p-q||\iii|}
  \end{equation*}
  for all $\iii \in \Sigma_*$. As $K^{-|p-q||\iii|} \le \mu([\iii])^{q-p} \le K^{|p-q||\iii|}$ for all $\iii \in \Sigma_*$, we see that
  \begin{equation} \label{eq:psi-upperbound}
  \begin{split}
    \psi^{q,s}(\iii) &\le \mu([\iii])^p \mu([\iii])^{q-p} \fii^t(\iii)^{1-q} \ualpha(q,s,t)^{(s-t)(1-q)|\iii|} \\ 
    &\le \psi^{p,t}(\iii) K^{|p-q||\iii|} \lalpha^{-t|p-q||\iii|} \ualpha(q,s,t)^{(s-t)(1-q)|\iii|}
  \end{split}
  \end{equation}
  and, similarly,
  \begin{equation} \label{eq:psi-lowerbound}
    \psi^{q,s}(\iii) \ge \psi^{p,t}(\iii) K^{-|p-q||\iii|} \lalpha^{t|p-q||\iii|} \lalpha(q,s,t)^{(s-t)(1-q)|\iii|}
  \end{equation}
  for all $\iii \in \Sigma_*$. It follows that
  \begin{align*}
    -|p-q|\log K &+ t|p-q|\log\lalpha + (s-t)(1-q)\log\lalpha(q,s,t) \\ 
    &\le P(\psi^{q,s}) - P(\psi^{p,t}) \\ 
    &\le |p-q|\log K - t|p-q|\log\lalpha + (s-t)(1-q)\log\ualpha(q,s,t)
  \end{align*}
  and hence the function $(q,s) \mapsto P(\psi^{q,s})$ is continuous on $\R \times [0,2)$. This proves (1). In particular, if $q < 1$ and $s>t$, then the above estimate shows that $P(\psi^{q,s}) - P(\psi^{q,t}) \leq (s-t)(1-q)\log\ualpha(q,s,t) = (s-t)(1-q)\log\ualpha < 0$, and if $q > 1$ and $s>t$, we get $P(\psi^{q,s}) - P(\psi^{q,t}) \geq (s-t)(1-q)\log\ualpha > 0$. These observations give (2)--(4). Notice that (5) follows immediately from (1).

  It remains to prove (6) and (7). Fix $p,q \in \R$ and $0 < \lambda < 1$. Let $s,t \in [0,2) \setminus \{1\}$ be such that $\lceil s \rceil = \lceil t \rceil$. Since $\fii^{\lambda t + (1-\lambda)s}(\iii) = \fii^t(\iii)^\lambda \fii^s(\iii)^{1-\lambda}$, we have
  \begin{align*}
    \mu([\iii])^q \fii^{\lambda t + (1-\lambda)s}(\iii)^{1-q} &= \mu([\iii])^q \fii^t(\iii)^{\lambda(1-q)} \fii^s(\iii)^{(1-\lambda)(1-q)} \\ 
    &= \Bigl( \mu([\iii])^q \fii^t(\iii)^{1-q} \Bigr)^\lambda \Bigl( \mu([\iii])^q \fii^s(\iii)^{1-q} \Bigr)^{1-\lambda}
  \end{align*}
  and
  \begin{align*}
    \psi^{\lambda p+(1-\lambda)q,s}(\iii) &= \mu([\iii])^{\lambda p+(1-\lambda)q} \fii^s(\iii)^{1-\lambda p-(1-\lambda)q} \\ 
    &= \Bigl( \mu([\iii])^p \fii^s(\iii)^{1-p} \Bigr)^\lambda \Bigl( \mu([\iii])^q \fii^s(\iii)^{1-q} \Bigr)^{1-\lambda}
  \end{align*}
  for all $\iii \in \Sigma_*$. Therefore, by H\"older's inequality, we see that
  \begin{equation*}
    \sum_{\iii \in \Sigma_n} \psi^{q,\lambda t + (1-\lambda)s}(\iii) \le \biggl( \sum_{\iii \in \Sigma_n} \psi^{q,t}(\iii) \biggr)^\lambda \biggl( \sum_{\iii \in \Sigma_n} \psi^{q,s}(\iii) \biggr)^{1-\lambda}
  \end{equation*}
  and
  \begin{equation*}
    \sum_{\iii \in \Sigma_n} \psi^{\lambda p+(1-\lambda)q,s}(\iii) \le \biggl( \sum_{\iii \in \Sigma_n} \psi^{p,s}(\iii) \biggr)^\lambda \biggl( \sum_{\iii \in \Sigma_n} \psi^{q,s}(\iii) \biggr)^{1-\lambda}
  \end{equation*}
  for all $n \in \N$. The claims follow by taking logarithms, dividing by $n$, and letting $n \to \infty$.
\end{proof}

Although the coordinate functions $q \mapsto P(\psi^{q,s})$ and $s \mapsto P(\psi^{q,s})$ are convex on certain intervals by Lemma \ref{thm:pressure-convex}(6) and (7), the function $(q,s) \mapsto P(\psi^{q,s})$ need not be convex. For example, if $0<\alpha_2<\alpha_1<1$ and
\begin{equation*}
  A =
  \begin{pmatrix}
    \alpha_1 & 0 \\ 
    0 & \alpha_2
  \end{pmatrix},
\end{equation*}
then the tuple $(A,\ldots,A) \in GL_2(\R)^N$ is dominated, and if $\mu \in \MM_\sigma(\Sigma)$ is the Bernoulli measure obtained from the probability vector $(\frac{1}{N},\ldots,\frac{1}{N})$, then
\begin{equation*}
  P(\psi^{q,s}) =
  \begin{cases}
    \log N^{1-q}\alpha_1^{(1-q)s}, &\text{if } 0 \le s < 1, \\ 
    \log N^{1-q}\alpha_1^{1-q}\alpha_2^{(1-q)(s-1)}, &\text{if } 1 \le s < 2.
  \end{cases}
\end{equation*}
In this case, the function $(q,s) \mapsto P(\psi^{q,s})$ is not convex because its Hessian is indefinite. Indeed, it is an antidiagonal matrix with
\begin{equation*}
  \begin{cases}
    -\log\alpha_1, &\text{if } 0 < s < 1, \\ 
    -\log\alpha_2, &\text{if } 1 < s < 2
  \end{cases}
\end{equation*}
in the antidiagonal.
    
Let us next study the differentiability of the pressure. We will first recall some basic facts from convex analysis. Let $U \subset \R$ be an open set and let $f \colon U \to \R$ be convex. It is well known that such a function is continuous. We say that $G \in \R$ is a \emph{sub-derivative} of $f$ at $x \in U$ if
\begin{equation*}
  f(y)-f(x) \ge G(y-x)
\end{equation*}
for all $y \in U$. It is straightforward to see that any sub-derivative is contained in $[f_-'(x),f_+'(x)]$, where $f_-'(x)$ and $f_+'(x)$ are the left and right derivatives of $f$ at $x$, respectively; see \cite[Theorem 23.2]{Rockafellar1970}. Therefore, $f$ is differentiable at $x$ if and only if it has only one sub-derivative at $x$; see \cite[Theorem 25.1]{Rockafellar1970}.

\begin{prop} \label{thm:pressure-gradient}
  Let $(A_1,\ldots,A_N) \in GL_2(\R)^N$ be a dominated tuple of contractive matrices and $\mu \in \MM(\Sigma)$ be a quasi-Bernoulli measure. If $(q_0,s_0) \in \R \times (0,\infty) \setminus \{1,2\}$ and $\nu$ is the equilibrium state for $\psi^{q_0,s_0}$, then the partial derivatives of $(q,s) \mapsto P(\psi^{q,s})$ are
  \begin{equation*}
    \partial_q P(\psi^{q,s_0})|_{q=q_0} = -h(\mu,\nu)-\Lambda(\fii^{s_0},\nu)
  \end{equation*}
  and
  \begin{equation*}
    \partial_s P(\psi^{q_0,s})|_{s=s_0} =
    \begin{cases}
      (1-q_0)\lambda_1(\nu), &\text{if } 0 < s_0 < 1, \\
      (1-q_0)\lambda_2(\nu), &\text{if } 1 < s_0 < 2, \\
      \frac12(1-q_0)(\lambda_1(\nu)+\lambda_2(\nu)), &\text{if } 2 < s_0 < \infty
    \end{cases}
  \end{equation*}
  provided that they exist.
\end{prop}

\begin{proof}
  We prove the result only for $(q_0,s_0) \in \R \times (0,2) \setminus \{1\}$; the case $(q_0,s_0) \in \R \times (2,\infty)$ is left to the reader. To simplify notation, we write $Q(q,s) = P(\psi^{q,s})$ for all $(q,s) \in \R \times (0,2) \setminus \{1\}$. Fix $(q_0,s_0) \in \R \times (0,2) \setminus \{1\}$ and let $\nu$ be the equilibrium state for $\psi^{q_0,s_0}$.

  Let us first assume that the partial derivative $\partial_q Q(q,s_0)|_{q=q_0}$ exists. Recall that, by Lemma \ref{thm:pressure-convex}(7), the function $q \mapsto Q(q,s_0)$ is convex. Since $\Lambda(\psi^{q,s},\nu) = -qh(\mu,\nu) + (1-q)\Lambda(\fii^s,\nu)$, we see that   
  \begin{align*}
    Q(q,s_0)-Q(q_0,s_0) &\geq h(\nu)+\Lambda(\psi^{q,s_0},\nu)-h(\nu)-\Lambda(\psi^{q_0,s_0},\nu) \\
    &= (-h(\mu,\nu)-\Lambda(\fii^{s_0},\nu))(q-q_0)
  \end{align*}
  for all $q \in \R$. Therefore, $-h(\mu,\nu)-\Lambda(\fii^{s_0},\nu)$ is a sub-derivative of the convex function $q \mapsto Q(q,s_0)$ at $q_0$. As the partial derivative $\partial_q Q(q,s_0)|_{q=q_0}$ exists, we have $\partial_q Q(q,s_0)|_{q=q_0} = -h(\mu,\nu)-\Lambda(\fii^{s_0},\nu)$.

  Let us then assume that the partial derivative $\partial_s Q(q_0,s)|_{s=s_0}$ exists. Recall that, by Lemma \ref{thm:pressure-convex}(6), the function $s \mapsto Q(q_0,s)$ is convex on connected components of $(0,2) \setminus \{1\}$. Since $\Lambda(\psi^{q,s},\nu) = -qh(\mu,\nu)+(1-q)\Lambda(\fii^s,\nu)$ and $\Lambda(\fii^s,\nu) = \Lambda(\fii^{s_0},\nu) + (s-s_0)\lambda_{\lceil s_0 \rceil}(\nu)$, we see that
  \begin{align*}
    Q(q_0,s)-Q(q_0,s_0) &\geq h(\nu)+\Lambda(\psi^{q_0,s},\nu)-h(\nu)-\Lambda(\psi^{q_0,s_0},\nu) \\
    &= (1-q_0)\lambda_{\lceil s_0 \rceil}(\nu)(s-s_0)
  \end{align*}
  for all $s \in (0,2) \setminus \{1\}$ with $\lceil s \rceil = \lceil s_0 \rceil$. Therefore, $(1-q_0)\lambda_{\lceil s_0 \rceil}(\nu)$ is a sub-derivative of the convex function $s \mapsto Q(q_0,s)$ at $s_0$. As the partial derivative $\partial_s Q(q_0,s)|_{s=s_0}$ exists, we have $\partial_s Q(q_0,s)|_{s=s_0} = (1-q_0)\lambda_{\lceil s_0 \rceil}(\nu)$.
\end{proof}

Let us next show that $(q,s) \mapsto P(\psi^{q,s})$ is differentiable on $\R \times (0,\infty) \setminus \{1,2\}$. This allows us to apply Proposition \ref{thm:pressure-gradient}. The proof uses tools from thermodynamic formalism. The following lemma allows us to employ Lemma \ref{thm:eq-convergence} in this setting.

\begin{lemma} \label{thm:psi-convergence}
  Let $(q_k,s_k)_{k \in \N}$ be a sequence of points in $\R \times (0,\infty) \setminus \{1,2\}$ converging to $(q,s) \in \R \times (0,\infty) \setminus \{1,2\}$. Then $\psi^{q_k,s_k}(\iii)^{1/|\iii|} \to \psi^{q,s}(\iii)^{1/|\iii|}$ uniformly in $\Sigma_*$ as $k \to \infty$.
\end{lemma}

\begin{proof}
  Using the notation from the proof of Lemma \ref{thm:pressure-convex}, the estimates \eqref{eq:psi-upperbound} and \eqref{eq:psi-lowerbound} imply
  \begin{align*}
    K^{-|q_k-q|} \lalpha^{s_k|q_k-q|} \lalpha(q,s,s_k)^{(s-s_k)(1-q)} &\le \biggl( \frac{\psi^{q,s}(\iii)}{\psi^{q_k,s_k}(\iii)} \biggr)^{1/|\iii|} \\ 
    &\le K^{|q_k-q|} \lalpha^{-s_k|q_k-q|} \ualpha(q,s,s_k)^{(s-s_k)(1-q)}
  \end{align*}
  for all $\iii \in \Sigma_*$ and $k \in \N$. This proves the claim.
\end{proof}

Before proving the differentiability, let us recall some further facts from convex analysis. Let $U \subset \R$ be an open set and let $f \colon U \to \R$ be convex. Let $D \subset U$ be the set of points where $f$ is differentiable. If $z_1,z_2 \in D$ and $x \in U$ such that $z_1 \le x \le z_2$, then, by \cite[Theorem 24.1]{Rockafellar1970}, $f'(z_1) \le G \le f'(z_2)$ for all sub-derivatives $G$ at $x$. It also follows that the set $U \setminus D$ is at most countable and $f'$ is continuous on $D$; see \cite[Theorem 25.3]{Rockafellar1970}. In particular, $D$ is dense in $U$ and if $f$ is differentiable on $U$, then it is continuously differentiable on $U$. 

\begin{prop} \label{thm:pressure-differentiable}
  Let $(A_1,\ldots,A_N) \in GL_2(\R)^N$ be a dominated tuple of contractive matrices and $\mu \in \MM(\Sigma)$ be a quasi-Bernoulli measure. Then the function $(q,s) \mapsto P(\psi^{q,s})$ is differentiable on $\R \times (0,\infty) \setminus \{1,2\}$.
\end{prop}

\begin{proof}
  We prove the result only on $\R \times (0,2) \setminus \{1\}$; the case $\R \times (2,\infty)$ is left to the reader. To simplify notation, we write $Q(q,s) = P(\psi^{q,s})$ for all $(q,s) \in \R \times (0,2) \setminus \{1\}$. To see that $Q$ is differentiable on $\R \times (0,2) \setminus \{1\}$, it suffices to show that both partial derivatives of $Q$ exist at each point of $\R \times (0,2) \setminus \{1\}$. Indeed, assuming this is the case, then using Proposition \ref{thm:pressure-gradient} combined with Lemmata \ref{thm:eq-convergence}, \ref{thm:almost-cont}, and \ref{thm:psi-convergence} it is straightforward to prove the continuity of the partial derivatives which then implies the differentiability. 

  Fix $(q_0,s_0) \in \R \times (0,2) \setminus \{1\}$. By Lemma \ref{thm:pressure-convex}(7), we know that the partial derivative $\partial_q Q(q,s_0)$ exists for all, except possibly at countably many points of $\R$. Relying on this, choose two sequences $(q_k^-)_{k \in \N}$ and $(q_k^+)_{k \in \N}$ of points in $\R$ with $q_k^- \uparrow q_0$ and $q_k^+ \downarrow q_0$ as $k \to \infty$ so that the partial derivatives $\partial_q Q(q,s_0)|_{q=q_k^-}$ and $\partial_q Q(q,s_0)|_{q=q_k^+}$ exist.

  Let $\nu_k^-$ and $\nu_k^+$ be the equilibrium states associated to $\psi^{q_k^-, s_0}$ and $\psi^{q_k^+, s_0}$, respectively. Then, by Lemmata \ref{thm:psi-convergence} and \ref{thm:eq-convergence}, any limit point of the sequences $(\nu_k^-)_{k\in\mathbb N}$ and $(\nu_k^+)_{k\in\mathbb N}$ must be the unique equilibrium state $\nu$ of $\psi^{q_0,s_0}$. Thus, $\nu_k^-\to \nu$ and $\nu_k^+\to \nu$ by the compactness of $\MM_\sigma(\Sigma)$. Let $G_1$ be a sub-derivative of $q \mapsto Q(q,s_0)$ at $q_0$. It follows from Proposition \ref{thm:pressure-gradient} that
  \begin{align*}
    -h(\mu,\nu_k^-)-\Lambda(\fii^{s_0},\nu_k^-) &= \partial_q Q(q,s_0)|_{q=q_k^-} \le G_1 \\ 
    &\le \partial_q Q(q,s_0)|_{q=q_k^+} = -h(\mu,\nu_k^+)-\Lambda(\fii^{s_0},\nu_k^+),
  \end{align*}
   where both bounds converge to the same value by Lemmata \ref{thm:almost-cont} and \ref{thm:eq-convergence}. Hence, $G_1 = \partial_q Q(q,s_0)|_{q=q_0}$.
   
   Now we show that the other partial derivative exists. By Lemma \ref{thm:pressure-convex}(6), we know that the partial derivative $\partial_s Q(q_0,s)$ exists for all, except possibly at countably many points on $[0,\infty)$. Relying on this, choose two sequences $(s_k^-)_{k \in \N}$ and $(s_k^+)_{k \in \N}$ of points in $(0,2) \setminus \{1\}$ with $\lceil s_k^- \rceil = \lceil s_k^+ \rceil = \lceil s_0 \rceil$, $s_k^- \uparrow s_0$, and $s_k^+ \downarrow s_0$ as $k \to \infty$ so that the partial derivatives $\partial_s Q(q_0,s_k^-)|_{s=s_k^-}$ and $\partial_s Q(q_0,s_k^+)|_{s=s_k^+}$ exist.

  Similarly as before, let $\eta_k^-$ and $\eta_k^+$ be the equilibrium states associated to $\psi^{q_0, s_k^-}$ and $\psi^{q_0, s_k^+}$, respectively, and notice that $\eta_k^-\to \nu$ and $\eta_k^+\to \nu$, where $\nu$ is the unique equilibrium state of $\psi^{q_0,s_0}$. Let $G_2$ be a sub-derivative of $s \mapsto Q(q_0,s)$ at $s_0$. It follows from Proposition \ref{thm:pressure-gradient} that
  \begin{align*}
    (1-q_0)\lambda_{\lceil s_k^- \rceil}(\eta_k^-) &= \partial_s Q(q_0,s)|_{s=s_k^-} \le G_2 \\ 
    &\le \partial_s Q(q_0,s)|_{s=s_k^+} = (1-q_0)\lambda_{\lceil s_k^+ \rceil}(\eta_k^+),
  \end{align*}
  where both bounds converge to the same value by Lemmata \ref{thm:almost-cont} and \ref{thm:eq-convergence}. Hence, $G_2 = \partial_s Q(q_0,s)|_{s=s_0}$ finishing the proof.
\end{proof}

\section{Symbolic Spectrum} \label{sec:symbolic-spectra}

For each $q \ne 1$ let $s_q \in [0,\infty)$ be as in Lemma \ref{thm:pressure-convex}(4). The \emph{symbolic $L^q$-spectrum} is the function $\tau \colon \R \to \R$ defined by $\tau(q)=(q-1)s_q$ for all $q \in \R \setminus \{1\}$ and $\tau(1)=0$. The \emph{symbolic multifractal spectrum} is the function $f \colon [0,\infty) \to \R$ defined by
\begin{equation*}
  f(s) = \sup\{\diml(\eta):\eta\in\MM_\sigma(\Sigma)\text{ such that }\diml(\mu,\eta)=s\}.
\end{equation*}
We say that $\tau$ and $f$ form a \emph{Legendre transform pair}  at $(q,s)$ if $f(s) = qs-\tau(q)$. In this section, we show that $\tau$ and $f$ form a Legendre transform pair at $(q,\tau'(q))$. To that end, following \cite[\S 3]{BarralFeng}, let us begin by verifying the existence of $\tau'$ and reveal its connection to the Lyapunov cross-dimension.

\begin{prop} \label{thm:spectrum-derivative}
  Let $(A_1,\ldots,A_N) \in GL_2(\R)^N$ be a dominated tuple of contractive matrices, $\mu \in \MM(\Sigma)$ be a fully supported quasi-Bernoulli measure, and $q \ne 1$. Then the symbolic $L^q$-spectrum $\tau$ is continuously differentiable on a neighborhood of $q$ with
  \begin{equation*}
    \tau'(q) =
    \begin{cases}
      -\frac{h(\mu,\nu)}{\lambda_1(\nu)}, &\text{if } 0 < \tau(q)/(q-1) < 1, \\ 
      1-\frac{h(\mu,\nu)+\lambda_1(\nu)}{\lambda_2(\nu)}, &\text{if } 1 < \tau(q)/(q-1) < 2, \\ 
      -\frac{2h(\mu,\nu)}{\lambda_1(\nu)+\lambda_2(\nu)}, &\text{if } 2 < \tau(q)/(q-1) < \infty,
    \end{cases}
  \end{equation*}
  where $\nu \in \MM_\sigma(\Sigma)$ is the equilibrium state for $\psi^{q,\tau(q)/(q-1)}$. In particular, if $\tau'(q)$ and $\tau(q)/(q-1)$ are contained in the same interval, $(0,1)$, $(1,2)$ or $(2,\infty)$, then
  \begin{equation*}
    \tau'(q) = \diml(\mu,\nu).
  \end{equation*}
\end{prop}

\begin{proof}
  Fix $q_0 \ne 1$ and, by Lemma \ref{thm:pressure-convex}(4), let $s_{q_0} \ge 0$ be such that $P(\psi^{q_0,s_{q_0}}) = 0$. Note that, by the definition of $\tau$, $s_{q_0} = \tau(q_0)/(q_0-1)$. Recalling Propositions \ref{thm:pressure-differentiable} and \ref{thm:pressure-gradient}, implicit differentiation gives
  \begin{equation*}
    \frac{\mathrm{d}s_q}{\mathrm{d}q}\bigg|_{q=q_0} = -\frac{\partial_q P(\psi^{q,s_{q_0}})|_{q=q_0}}{\partial_s P(\psi^{q_0,s})|_{s=s_{q_0}}} =
    \begin{cases}
      \frac{h(\mu,\nu)+\Lambda(\fii^{s_{q_0}},\nu)}{(1-q_0)\lambda_1(\nu)}, &\text{if } 0<s_{q_0}<1, \\ 
      \frac{h(\mu,\nu)+\Lambda(\fii^{s_{q_0}},\nu)}{(1-q_0)\lambda_2(\nu)}, &\text{if } 1<s_{q_0}<2, \\ 
      \frac{h(\mu,\nu)+\Lambda(\fii^{s_{q_0}},\nu)}{\frac12(1-q_0)(\lambda_1(\nu)+\lambda_2(\nu))}, &\text{if } 2<s_{q_0}<\infty.
    \end{cases}
  \end{equation*}
  Since $\tau'(q_0) = s_{q_0}+(q_0-1)\frac{\mathrm{d}s_q}{\mathrm{d}q}|_{q=q_0}$, we get
  \begin{equation} \label{eq:spectrum-derivative1}
    \tau'(q_0) =
    \begin{cases}
      s_{q_0}-\frac{h(\mu,\nu)+\Lambda(\fii^{s_{q_0}},\nu)}{\lambda_1(\nu)}, &\text{if } 0<s_{q_0}<1, \\ 
      s_{q_0}-\frac{h(\mu,\nu)+\Lambda(\fii^{s_{q_0}},\nu)}{\lambda_2(\nu)}, &\text{if } 1<s_{q_0}<2, \\ 
      s_{q_0}-\frac{h(\mu,\nu)+\Lambda(\fii^{s_{q_0}},\nu)}{\frac12(\lambda_1(\nu)+\lambda_2(\nu))}, &\text{if } 2<s_{q_0}<\infty.
    \end{cases}
  \end{equation}
  The first claim follows from this by a simple calculation. If both $\tau'(q_0)$ and $s_{q_0}$ are contained in the same interval, then the first claim and the version of \eqref{eq:diml-min} for the Lyapunov cross-dimension imply the second claim.
\end{proof}

Let us next determine the Lyapunov dimension of the equilibrium state we are considering.

\begin{theorem}\label{thm:dim-tau-derivative}
  Let $(A_1,\ldots,A_N) \in GL_2(\R)^N$ be a dominated tuple of contractive matrices, $\mu \in \MM(\Sigma)$ be a fully supported quasi-Bernoulli measure, and $q \ne 1$. If $\tau(q)/(q-1)$ and $q\tau'(q)-\tau(q)$ are contained in the same interval, $(0,1)$, $(1,2)$ or $(2,\infty)$, then
  \begin{equation*}
    \diml(\nu) = q\tau'(q)-\tau(q),
  \end{equation*}
  where $\nu \in \MM_\sigma(\Sigma)$ is the equilibrium state for $\psi^{q,\tau(q)/(q-1)}$.
\end{theorem}

\begin{proof}
  Recall that $s_{q} = \tau(q)/(q-1) \ge 0$ is such that $P(\psi^{q,s_{q}}) = 0$. Therefore,
  \begin{equation} \label{eq:eq-state-calc2}
  \begin{split}
    0 &= P(\psi^{q,s_q}) = h(\nu) + \Lambda(\psi^{q,s_q},\nu) \\
    &= h(\nu) - qh(\mu,\nu) + (1-q)\Lambda(\fii^{s_q},\nu)
  \end{split}
  \end{equation}
  If $s_q \in (0,1)$, then \eqref{eq:eq-state-calc2} and Proposition \ref{thm:spectrum-derivative} imply
  \begin{equation*}
    -\frac{h(\nu)}{\lambda_1(\nu)} = -q\frac{h(\mu,\nu)}{\lambda_1(\nu)} - \tau(q) = q\tau'(q) - \tau(q).
  \end{equation*}
  Since $q\tau'(q) - \tau(q) \in (0,1)$, we have, by \eqref{eq:diml-min},
  \begin{equation*}
    \diml(\nu) = -\frac{h(\nu)}{\lambda_1(\nu)} = q\tau'(q) - \tau(q).
  \end{equation*}
  Furthermore, if $s_q \in (1,2)$, then \eqref{eq:eq-state-calc2} and Proposition \ref{thm:spectrum-derivative} give
  \begin{equation*}
    1 - \frac{h(\nu)+\lambda_1(\nu)}{\lambda_2(\nu)} = q\biggl(1-\frac{h(\mu,\nu)+\lambda_1(\nu)}{\lambda_2(\nu)}\biggr) - \tau(q) = q\tau'(q)-\tau(q).
  \end{equation*}
  Since $q\tau'(q) - \tau(q) \in (1,2)$, we have, by \eqref{eq:diml-min},
  \begin{equation*}
    \diml(\nu) = 1 - \frac{h(\nu)+\lambda_1(\nu)}{\lambda_2(\nu)} = q\tau'(q) - \tau(q).
  \end{equation*}
  Finally, if $s_q \in (2,\infty)$, then \eqref{eq:eq-state-calc2} and Proposition \ref{thm:spectrum-derivative} imply
  \begin{equation*}
    -\frac{2h(\nu)}{\lambda_1(\nu)+\lambda_2(\nu)} = -q\frac{2h(\mu,\nu)}{\lambda_1(\nu)+\lambda_2(\nu)} - \tau'(q) = q\tau'(q)-\tau(q).
  \end{equation*}
  Since $q\tau'(q) - \tau(q) \in (2,\infty)$, we have, by \eqref{eq:diml-min},
  \begin{equation*}
    \diml(\nu) = -\frac{2h(\nu)}{\lambda_1(\nu)+\lambda_2(\nu)} = q\tau'(q) - \tau(q)
  \end{equation*}
  and the proof is finished.
\end{proof}

In the main result of this section, we completely characterize the behavior of the symbolic multifractal spectrum.

\begin{theorem}\label{LegendreThm}
  Let $(A_1,\ldots,A_N) \in GL_2(\R)^N$ be a dominated tuple of contractive matrices, $\mu \in \MM(\Sigma)$ be a fully supported quasi-Bernoulli measure, and $q \ne 1$. If $\tau'(q)$, $\tau(q)/(q-1)$, and $q\tau'(q)-\tau(q)$ are contained in the same interval, $(0,1)$, $(1,2)$ or $(2,\infty)$, then $\tau'(q) \le \tau(q)/(q-1)$ and
  \begin{align*}
    f(\tau'(q)) &= \sup\{\diml(\eta):\eta\in\MM_\sigma(\Sigma)\text{ such that }\diml(\mu,\eta)=\tau'(q)\} \\ 
    &= \diml(\nu) = q\tau'(q)-\tau(q),
  \end{align*}
  where $\nu \in \MM_\sigma(\Sigma)$ is the equilibrium state for $\psi^{q,\tau/(q-1)}$, i.e.\ $\tau$ and $f$ form a Legendre transform pair at $(q,\tau'(q))$.
\end{theorem}

\begin{proof}
  By Proposition \ref{thm:spectrum-derivative} and Theorem \ref{thm:dim-tau-derivative}, the equilibrium state $\nu$ satisfies $\diml(\mu,\nu) = \tau'(q)$ and $\diml(\nu) = q\tau'(q)-\tau(q)$. Since $\diml(\nu) \le \diml(\mu,\nu)$ by Lemma \ref{thm:entropy-crossentropy}, we have $\tau'(q) \le \tau(q)/(q-1)$. Furthermore, if $\eta \in \MM_\sigma(\Sigma)$ is such that $\diml(\mu,\eta) = \tau'(q)$, then we have
  \begin{align*}
    0 &= P(\psi^{q,s_q}) \ge h(\eta) + \Lambda(\psi^{q,s_q},\eta) \\
    &= h(\eta) - qh(\mu,\eta) + (1-q)\Lambda(\fii^{s_q},\eta) \\
    &= h(\eta) + q\Lambda(\fii^{\tau'(q)},\eta) + (1-q)\Lambda(\fii^{s_q},\eta) \\
    &= h(\eta)+\Lambda(\fii^{q\tau'(q)-\tau(q)},\eta)
  \end{align*}
  and, consequently, $\diml(\eta) \le q\tau'(q)-\tau(q)$. Therefore, the supremum in the claim is attained with $\diml(\nu)$ and the proof is finished.
\end{proof}

\section{Exact Dimensionality} \label{sec:exact-dimensionality}

Given an affine IFS $(\fii_1,\ldots, \fii_N)$, where $\fii_i(x) = A_ix+v_i$, let $X$ be the planar self-affine set satisfying \eqref{eq:self-affine-set-def}. The \emph{canonical projection} $\pi\colon \Sigma \to X$ is defined by
\begin{equation*}
  \pi(\iii) = \lim_{n\to\infty} \fii_{\iii|_n}(0) = \sum_{n=1}^\infty A_{\iii|_{n-1}} v_{i_n}
\end{equation*}
for all $\iii = i_1i_2\cdots \in \Sigma$. Here $\fii_\iii = \fii_{i_1} \circ \cdots \circ \fii_{i_n}$ for all $\iii = i_1 \cdots i_n \in \Sigma_n$ and $n \in \N$. It is easy to see that $\pi(\Sigma) = X$. We say that $X$ satisfies the \emph{strong separation condition} if $\fii_i(X) \cap \fii_j(X) = \emptyset$ whenever $i \ne j$. As $\pi([\iii]) = \fii_\iii(X)$ for all $\iii \in \Sigma_*$, the strong separation condition is characterized by the property that the canonical projection $\pi$ is one-to-one. We adopt the convention that whenever we speak about a self-affine set, it is automatically accompanied by the tuple of affine maps that defines it. This allows us to write, for example, ``a self-affine set $X$ is dominated,'' which is understood to mean that ``the associated tuple $\A$ of matrices is dominated.''

Let $\nu \in \MM_\sigma(\Sigma)$ be a quasi-Bernoulli measure. If $X$ is dominated, then, by \cite[Theorem 2.6]{BaranyKaenmaki2017} and \cite[Theorem 2.2]{Rossi2014},
\begin{equation} \label{eq:exact-dim}
  \dimloc(\pi_*\nu,\pi(\iii)) = \ldimh(\pi_*\nu) \le \diml(\nu)
\end{equation}
for $\nu$-almost all $\iii \in \Sigma$. In other words, canonical projections of quasi-Bernoulli measures, which we call quasi-Bernoulli measures on $X$, are \emph{exact-dimensional}. We say that $\A = (A_1,\ldots,A_N) \in GL_2(\R)^N$ is \emph{irreducible} if there is no line in $\R^2$ which is invariant under all of the matrices in $\A$. If $\A$ is dominated such that there is a finite set of lines in $\R^2$ invariant under all of the matrices in $\A$, then, by \cite[Lemma 2.10]{BaranyKaenmakiYu2021}, $\A$ is not irreducible. Therefore, if $X$ is dominated and irreducible, and satisfies the strong separation condition, then, by \cite[Theorem 1.2 and the associated footnote]{BaranyHochmanRapaport},
\begin{equation} \label{eq:BHR}
  \ldimh(\pi_*\nu) = \min\{2,\diml(\nu)\}
\end{equation}
for all quasi-Bernoulli measures $\nu \in \MM_\sigma(\Sigma)$. As observed in \cite[Theorem 2.18]{BaranyKaenmakiYu2021}, \eqref{eq:BHR} holds also when $X$ is dominated and $X_F$ is not a singleton.

In this section, by generalizing the approach used in \cite[\S 4.1]{KemptonFalconer}, we study the local dimension of a quasi-Bernoulli measure on a self-affine set with respect to another measure. We first recall the following lemma.

\begin{lemma} \label{thm:falconer-kempton-lemma}
  Let $X$ be a dominated planar self-affine set satisfying the strong separation condition. If $\mu \in \MM(\Sigma)$ is fully supported quasi-Bernoulli, then there are constants $C \ge 1$ and $\roo_2 \ge \roo_1 > 0$ such that
  \begin{align*}
    C^{-1}&\pi_*\mu(B(\pi(\iii),\roo_1\alpha_2(\iii|_n))) \\
    &\le \mu([\iii|_n]) \bigl(\proj_{(A_{\overleftarrow{\iii|_n}}^{-1}V)^\bot}\bigr)_*(\pi_*\mu)\biggl(\proj_{(A_{\overleftarrow{\iii|_n}}^{-1}V)^\bot}\biggl(B\biggl(\pi(\sigma^n\iii),\frac{\alpha_2(\iii|_n)}{\alpha_1(\iii_n)}\biggr)\biggr)\biggr) \\
    &\le C\pi_*\mu(B(\pi(\iii),\roo_2\alpha_2(\iii|_n)))
  \end{align*}
  for all $\iii \in \Sigma$, $V \in \overline{\RP \setminus \CC}$, and $n \in \N$.
\end{lemma}

\begin{proof}
  The proof is the same as that of \cite[Lemma 4.2]{KemptonFalconer} where the result is formulated in the case of positive matrices.
\end{proof}

The main result of this section is the following theorem.

\begin{theorem} \label{thm:exact-dim}
  Let $X$ be a dominated planar self-affine set satisfying the strong separation condition. If $\mu \in \MM(\Sigma)$ is fully supported quasi-Bernoulli and $\nu \in \MM_\sigma(\Sigma)$ is ergodic and quasi-Bernoulli, and assume that 
  \[
  \dimloc((\proj_{V^\bot})_*(\pi_*\mu),\proj_{V^\bot}(\iii))=1
  \]
  for $\nu\times \nu_F$-almost all $(\iii,V) \in \Sigma\times\mathbb R\mathbb P^1$, then
  \begin{equation*}
    \dimloc(\pi_*\mu,\pi(\iii)) = \min\{ 2, \diml(\mu,\nu) \}
  \end{equation*}
  for $\nu$-almost all $\iii \in \Sigma$.
\end{theorem}

\begin{proof}
  Fix $\eps > 0$ and write
  \begin{equation*}
    d_r(\iii,V) = \frac{\log (\proj_{V^\bot})_*(\pi_*\mu)(\proj_{V^\bot}(B(\pi(\iii),r)))}{\log r}
  \end{equation*}
  for all $\iii \in \Sigma$, $V \in \RP$, and $r > 0$. By the assumption of the theorem, we have $\lim_{r \downarrow 0} d_r(\iii,V) = 1$ for $\nu \times \nu_F$-almost all $(\iii,V) \in \Sigma \times \RP$. Let $\kappa>0$ and observe that, by Egorov's theorem, there exists $r_0>0$ such that the set
  \begin{equation*}
    \Gamma_\kappa = \{(\iii,V) \in \Sigma \times \RP : |d_r(\iii,V) - 1| < \eps \text{ for all } 0<r<r_0\}
  \end{equation*}
  satisfies $\nu \times \nu_F(\Gamma_\kappa) > 1-\kappa$. Since $\nu$ is quasi-Bernoulli, there exists a $T$-invariant ergodic measure $\nu^\ast$ on $\Sigma \times \RP$ which is equivalent to $\nu \times \nu_F$. Therefore, $\nu^\ast(\Gamma_\kappa) > 1-\kappa$ and Birkhoff's ergodic theorem implies that
  \begin{align*}
    \lim_{k \to \infty} \frac{1}{k} \#\{n \in \{1,\ldots,k\} &: |d_r(T^n(\iii,V))-1| < \eps \text{ for all } 0<r<r_0\} \\ 
    &= \lim_{k \to \infty} \frac{1}{k} \#\{n \in \{1,\ldots,k\} : T^n(\iii,V) \in \Gamma_\kappa\} \\ 
    &= \nu^\ast(\Gamma_\kappa) > 1-\kappa
  \end{align*}
  for $\nu^\ast$-almost all $(\iii,V) \in \Sigma \times \RP$ and hence for $\nu \times \nu_F$-almost all $(\iii,V) \in \Sigma \times \RP$. Write $r_n(\iii) = \alpha_2(\iii|_n)/\alpha_1(\iii|_n)$ for all $\iii \in \Sigma$ and $n \in \N$, and note that, by domination, there exists $n_0 \in \N$ such that $r_n(\iii) < r_0$ for all $\iii \in \Sigma$ and $n \ge n_0$. By letting $\kappa \downarrow 0$, it follows from the above estimate that
  \begin{equation} \label{eq:density-one}
    \lim_{k \to \infty} \frac{1}{k} \#\{n \in \{1,\ldots,k\} : |d_{r_n(\iii)}(T^n(\iii,V))-1| < \eps\} = 1
  \end{equation}
  for $\nu \times \nu_F$-almost all $(\iii,V) \in \Sigma \times \RP$.

  Let $(\iii,V) \in \Sigma \times \RP$ be such that \eqref{eq:density-one} is satisfied. We may choose a strictly increasing sequence $(n_k)_{k \in \N}$ of integers such that $\lim_{k \to \infty} n_k/k = 1$ and
  \begin{equation} \label{eq:dim-quantify}
    |d_{r_{n_k}(\iii)}(T^{n_k}(\iii,V))-1| < \eps
  \end{equation}
  for all $k \in \N$. Since $\nu$ is ergodic and $\mu$, $\alpha_1$, and $1/\alpha_2$ are sub-multiplicative, Lemma \ref{thm:fekete} implies that
  \begin{equation} \label{eq:kingman-versions}
  \begin{split}
    \lambda_1(\nu) &= \lim_{k \to \infty} \frac{1}{n_k} \log\alpha_1(\iii|_{n_k}), \\ 
    \lambda_2(\nu) &= \lim_{k \to \infty} \frac{1}{n_k} \log\alpha_2(\iii|_{n_k}), \\
    h(\mu,\nu) &= -\lim_{k \to \infty} \frac{1}{n_k} \log\mu([\iii|_{n_k}])
  \end{split}
  \end{equation}
  for $\nu$-almost all $\iii \in \Sigma$.

  Let $\roo_2>0$ be as in Lemma \ref{thm:falconer-kempton-lemma} and notice that, since $\alpha_2$ is almost-multiplicative by domination, there exists a constant $C \ge 1$ such that
  \begin{align*}
    1 &\ge \frac{\log \roo_2\alpha_2(\iii|_{n_{k+1}})}{\log \roo_2\alpha_2(\iii|_{n_k})} \ge \frac{\log \roo_2\alpha_2(\iii|_{n_k}) + \log \alpha_2(\sigma^{n_k}\iii|_{n_{k+1}-n_k}) + \log C}{\log \roo_2\alpha_2(\iii|_{n_k})} \\ 
    &\ge 1 + \frac{(n_{k+1}-n_k)\log C\max_{i \in \{1,\ldots,N\}}\alpha_2(i)}{n_k\log \roo_2\min_{i \in \{1,\ldots,N\}}\alpha_2(i)} + \frac{\log C}{\log \roo_2\alpha_2(\iii|_{n_k})}.
  \end{align*}
  Since $\lim_{k \to \infty} n_k/k = 1$, we see that $\lim_{k \to \infty} (n_{k+1}-n_k)/n_k = 0$ and hence,
  \begin{equation*}
    \lim_{k \to \infty} \frac{\log \roo_2\alpha_2(\iii|_{n_{k+1}})}{\log \roo_2\alpha_2(\iii|_{n_k})} = 1
  \end{equation*}
  and
  \begin{equation*}
    \udimloc(\pi_*\mu,\pi(\iii)) = \limsup_{k \to \infty} \frac{\log \pi_*\mu(B(\pi(\iii),\roo_2\alpha_2(\iii|_{n_k})))}{\log \roo_2\alpha_2(\iii|_{n_k})}.
  \end{equation*}
  Therefore, by Lemma \ref{thm:falconer-kempton-lemma}, \eqref{eq:kingman-versions}, and \eqref{eq:dim-quantify}, there exists a constant $C \ge 1$ such that
  \begin{align*}
    &\udimloc(\pi_*\mu,\pi(\iii)) \le \limsup_{k \to \infty} \Biggl( \frac{\log C^{-1}\mu([\iii|_{n_k}])}{\log \roo_2\alpha_2(\iii|_{n_k})} \\ 
    &\qquad\quad+ \frac{\log \bigl(\proj_{(A_{\overleftarrow{\iii|_{n_k}}}^{-1}V)^\bot}\bigr)_*(\pi_*\mu)\bigl(\proj_{(A_{\overleftarrow{\iii|_{n_k}}}^{-1}V)^\bot}\bigl(B\bigl(\pi(\sigma^{n_k}\iii),\frac{\alpha_2(\iii|_{n_k})}{\alpha_1(\iii|_{n_k})}\bigr)\bigr)\bigr)}{\log \roo_2\alpha_2(\iii|_{n_k})} \Biggr) \\ 
    &\quad\le \limsup_{k \to \infty} \frac{\log C^{-1}\mu([\iii|_{n_k}])}{\log \roo_2\alpha_2(\iii|_{n_k})} + \limsup_{k \to \infty} d_{r_{n_k}(\iii)}(T^{n_k}(\iii,V)) \frac{\log \frac{\alpha_2(\iii|_{n_k})}{\alpha_1(\iii|_{n_k})}}{\log \roo_2\alpha_2(\iii|_{n_k})} \\ 
    &\quad\le \frac{-h(\mu,\nu)}{\lambda_2(\nu)} + (1+\eps)\frac{\lambda_2(\nu)-\lambda_1(\nu)}{\lambda_2(\nu)}.
  \end{align*}
  By letting $\eps \downarrow 0$, we see that
  \begin{equation*}
    \udimloc(\pi_*\mu,\pi(\iii)) \le \min\biggl\{2, 1 - \frac{h(\mu,\nu)+\lambda_1(\nu)}{\lambda_2(\nu)} \biggr\}=\min\{2, \diml(\mu,\nu)\}
  \end{equation*}
  for $\nu$-almost all $\iii \in \Sigma$. As a similar argument shows that the right-hand side above is a lower bound for $\ldimloc(\pi_*\mu,\pi(\iii))$ for $\nu$-almost all $\iii \in \Sigma$, we have finished the proof.
\end{proof}

\section{Multifractal Formalism} \label{sec:multifractal-formalism}

In this section, we extend the symbolic multifractal formalism results of Section \ref{sec:symbolic-spectra} to obtain results in terms of the Hausdorff dimension of the sets
\begin{equation*}
  X_s = X(\mu,s) = \{x\in X : \dimloc(\mu,x)=s\}.
\end{equation*}
This section contains several short propositions, some of which are quite similar; the proofs of Theorems \ref{1.2} and \ref{1.3} are presented after them.

\begin{prop} \label{thm:upper-bound-small-q}
   Let $X$ be a dominated planar self-affine set satisfying the strong separation condition, $\mu \in \MM(\Sigma)$ be a fully supported quasi-Bernoulli measure, $0<q<1$, and $s \ge 0$. Then 
  \begin{equation*}
      \dimh(X_s) \le qs-\tau(q).
  \end{equation*}
\end{prop}

\begin{proof}
  In general, the result follows from \cite[Theorem 6.2(a)]{Falconer1999} and \cite[Proposition 2.5(iv)]{OLSEN199582}. Our version follows directly from \cite[Lemmas 5.1 and 5.3]{BarralFeng} and the comments after the lemmas. The restriction $0<q<1$ arises because, in the language of \cite[Lemma 5.1]{BarralFeng}, the inequality between $\tau(\mu,q)$ and $\tau(q)$ changes sign at $q=1$.
\end{proof}

The arguments of the following proposition are an adjustment of those in \cite[Theorem 4.1]{Lau}.

\begin{prop} \label{thm:upper-bound-large-q}
  Let $X$ be a dominated planar self-affine set satisfying the strong separation condition, $\mu \in \MM(\Sigma)$ be a fully supported quasi-Bernoulli measure, and $q > 1$. Suppose that one of the following two conditions is satisfied:
  \begin{enumerate}[(1)]
    \item \label{it:upper-bound-large-q-assumption-1}
    It holds that $0 \le \tau(q)/(q-1) \le 1$ and there exist constants $\roo_3>0$ and $C \ge 1$ such that
    \begin{equation*}
      \pi_*\mu(B(\pi(\iii),\roo_3\alpha_1(\iii|_n))) \le C\mu([\iii|_n])
    \end{equation*}
    for all $\iii \in \Sigma$ and $n \in \N$. 
    \item \label{it:upper-bound-large-q-assumption-2}
    It holds that $1 \le \tau(q)/(q-1) \le 2$ and there exists a constant $C \ge 1$ such that
    \begin{equation*}
      (\proj_{V^\bot})_*(\pi_*\mu)(\proj_{V^\bot}(B(x,r))) \le Cr
    \end{equation*}
    for all $x \in \R^2$, $V \in X_F$, and $r>0$.
  \end{enumerate}
  Then
  \begin{equation*}
    \dimh(X_s) \le qs-\tau(q)
  \end{equation*}
  for all $s \ge 0$.
\end{prop}

\begin{proof}
  Fix $q > 1$ and, recalling Lemma \ref{thm:pressure-convex}(4), let $s_q = \tau(q)/(q-1) \ge 0$ be such that $P(\psi^{q,s_q})=0$. Observe that if the assumption \ref{it:upper-bound-large-q-assumption-1} holds then $s_q \in (0,1)$, whereas if the assumption \ref{it:upper-bound-large-q-assumption-2} holds then $s_q \in (1,2)$. Let $\kappa > 0$ and notice that the potential $\alpha_m^\kappa\psi^{q,s_q}$ is almost-multiplicative for both $m \in \{1,2\}$. Under the assumption \ref{it:upper-bound-large-q-assumption-1}, we choose $m=1$, and under the assumption \ref{it:upper-bound-large-q-assumption-2}, we let $m=2$. If $\nu \in \MM_\sigma(\Sigma)$ is the equilibrium state for $\alpha_m^\kappa\psi^{q,s_q}$, then
  \begin{equation*}
      P(\alpha_m^\kappa\psi^{q,s_q}) = h(\nu)+\Lambda(\psi^{q,s_q},\nu) + \kappa\lambda_m(\nu) \le P(\psi^{q,s_q})+\kappa\lambda_m(\nu) < 0.
  \end{equation*}
  Fix $\kappa\lambda_2(\nu) \le \kappa\lambda_1(\nu) < \gamma < 0$ and observe that there is $n_\kappa \in \N$ such that
  \begin{equation} \label{eq:upper-bound-large-q-pressure}
    \sum_{\iii \in \Sigma_n} \alpha_m(\iii)^\kappa \mu([\iii])^q \fii^{s_q}(\iii)^{1-q} < e^{n\gamma}
  \end{equation}
  for all $n \ge n_\kappa$ and both $m \in \{1,2\}$.

  Fix $\delta > 0$ and choose $n_{\delta,\kappa} \ge n_\kappa$ such that $\alpha_m(\iii) < \delta$ whenever $|\iii| \ge n_{\delta,\kappa}$. Under the assumption \ref{it:upper-bound-large-q-assumption-1}, we denote $\roo=\roo_3>0$, and under the assumption \ref{it:upper-bound-large-q-assumption-2}, let $\roo>0$ be $\roo_1$ from Lemma \ref{thm:falconer-kempton-lemma}. Fix $s \ge 0$ and $\eps>0$ and notice that for each $\iii \in \pi^{-1}(X_s)$ there is $n_\iii \in \N$ such that
  \begin{equation*}
    \frac{\log \pi_*\mu(B(\pi(\iii),\roo\alpha_m(\iii|_n)))}{\log \alpha_m(\iii|_n)} < s + \eps
  \end{equation*}
  for all $n \ge n_\iii$. Therefore, the family
  \begin{align*}
    \{B(\pi(\iii),\roo\alpha_m(\iii|_n)) : \;&\iii \in \Sigma \text{ and } n \ge n_{\delta,\kappa} \text{ are such that} \\ 
    &\pi_*\mu(B(\pi(\iii),\roo\alpha_m(\iii|_n))) > \alpha_m(\iii|_n)^{s + \eps} \}
  \end{align*}
  is a Vitali covering of $X_s$. By the Vitali covering theorem (see e.g.\ \cite[Theorem 1.10]{Fractalsets}), the above family contains a countable sub-family $\VV$ of pairwise disjoint balls satisfying
  \begin{equation*}
    \sum_{B \in \VV} \diam(B)^t = \infty \quad \text{or} \quad \HH^t\biggl(X_s \setminus \bigcup_{B \in \VV} B\biggr) = 0,
  \end{equation*}
  where $t = (s+\eps)q - \tau(q) + \kappa$. If we can show that there is a constant $M \in \R$ not depending on $\delta$ such that
  \begin{equation} \label{eq:upper-bound-large-q-sum-finite}
    \sum_{B \in \VV} \diam(B)^t \le M < \infty,
  \end{equation}
  then
  \begin{equation*}
    \HH^t_\delta(X_s) \le \HH^t_\delta\biggl(X_s \cap \bigcup_{B \in \VV} B\biggr) + \HH^t_\delta\biggl(X_s \setminus \bigcup_{B \in \VV} B\biggr) \le \HH^t_\delta\biggl(\bigcup_{B \in \VV} B\biggr) \le M.
  \end{equation*}
  By letting $\delta \downarrow 0$, we get $\HH^t(X_s) \le M < \infty$ and $\dimh(X_s) \le t = (s+\eps)q - \tau(q) + \kappa$. The claim follows by letting $\eps,\kappa \downarrow 0$.

  Let us now prove \eqref{eq:upper-bound-large-q-sum-finite}. If $\jjj \in \Sigma_n$ and $n \ge n_{\delta,\kappa}$, then we write
  \begin{equation*}
    \VV_\jjj = \{B(\pi(\iii_k),\roo\alpha_m(\jjj)) \in \VV : k \in \{1,\ldots,\# \VV_\jjj\}\},
  \end{equation*}
  Since the elements of $\VV$ are pairwise disjoint, a simple volume argument shows that, under the assumption \ref{it:upper-bound-large-q-assumption-1}, $\# \VV_\jjj \le c$, and under the assumption \ref{it:upper-bound-large-q-assumption-2}, $\# \VV_\jjj \le c\frac{\alpha_1(\jjj)}{\alpha_2(\jjj)}$ for some constant $c>0$ not depending on $\jjj$. Notice that, as $q>0$, the definition of $\VV$ yields
  \begin{equation} \label{eq:upper-bound-large-q-calc1}
  \begin{split}
    \sum_{\jjj \in \Sigma_n} \sum_{B \in \VV_\jjj} \diam(B)^t &\le \sum_{\jjj \in \Sigma_n} \sum_{B \in \VV_\jjj} \diam(B)^t \biggl(\frac{\pi_*\mu(B)}{\alpha_m(\jjj)^{s + \eps}}\biggr)^q \\ 
    &= (2\roo)^t \sum_{\jjj \in \Sigma_n} \sum_{B \in \VV_\jjj} \alpha_m(\jjj)^{(s+\eps)q - \tau(q) + \kappa} \biggl(\frac{\pi_*\mu(B)}{\alpha_m(\jjj)^{s + \eps}}\biggr)^q \\ 
    &= (2\roo)^t \sum_{\jjj \in \Sigma_n} \sum_{B \in \VV_\jjj} \alpha_m(\jjj)^{-\tau(q)+\kappa} \pi_*\mu(B)^q
  \end{split}
  \end{equation}
  for all $n \ge n_{\delta,\kappa}$.

  By \eqref{eq:upper-bound-large-q-calc1}, the assumption \ref{it:upper-bound-large-q-assumption-1}, the fact that $\#\VV_\jjj \le c$, and \eqref{eq:upper-bound-large-q-pressure}, we have
  \begin{align*}
    \sum_{\jjj \in \Sigma_n} \sum_{B \in \VV_\jjj} \diam(B)^t &\le (2\roo)^t \sum_{\jjj \in \Sigma_n} \sum_{k=1}^{\#\VV_\jjj} \alpha_1(\jjj)^{-\tau(q)+\kappa} \pi_*\mu(B(\pi(\iii_k),\roo\alpha_1(\jjj)))^q \\ 
    &\le C^q(2\roo)^t \sum_{\jjj \in \Sigma_n} \sum_{k=1}^{\#\VV_\jjj} \alpha_1(\jjj)^{-\tau(q)+\kappa} \mu([\jjj])^q \\ 
    &\le cC^q(2\roo)^t \sum_{\jjj \in \Sigma_n} \alpha_1(\jjj)^\kappa \mu([\jjj])^q \fii^{s_q}(\jjj)^{1-q} \\ 
    &\le cC^q(2\roo)^t e^{n\gamma}.
  \end{align*}
  On the other hand, by Lemma \ref{thm:falconer-kempton-lemma} and the assumption \ref{it:upper-bound-large-q-assumption-2}, there are constants $\tilde C,C \ge 1$ such that
  \begin{equation} \label{eq:upper-bound-large-q-calc2}
  \begin{split}
    \pi_*\mu(&B(\pi(\iii),\roo\alpha_2(\iii|_n))) \\
    &\le \tilde C\mu([\iii|_n]) \bigl(\proj_{V^\bot}\bigr)_*(\pi_*\mu)\biggl(\proj_{V^\bot}\biggl(B\biggl(\pi(\sigma^n\iii),\frac{\alpha_2(\iii|_n)}{\alpha_1(\iii|_n)}\biggr)\biggr)\biggr) \\ 
    &\le C\mu([\iii|_n]) \frac{\alpha_2(\iii|_n)}{\alpha_1(\iii|_n)}
  \end{split}
  \end{equation}
  for all $\iii \in \Sigma$, $V \in X_F$, and $n \in \N$. Therefore, by \eqref{eq:upper-bound-large-q-calc1}, \eqref{eq:upper-bound-large-q-calc2}, the fact that $\#\VV_\jjj \le c\frac{\alpha_1(\jjj)}{\alpha_2(\jjj)}$, and \eqref{eq:upper-bound-large-q-pressure}, we similarly get
  \begin{align*}
    \sum_{\jjj \in \Sigma_n} \sum_{B \in \VV_\jjj} \diam(B)^t &\le (2\roo)^t \sum_{\jjj \in \Sigma_n} \sum_{k=1}^{\#\VV_\jjj} \alpha_2(\jjj)^{-\tau(q)+\kappa} \pi_*\mu(B(\pi(\iii_k),\roo\alpha_2(\jjj)))^q \\ 
    &\le C^q(2\roo)^t \sum_{\jjj \in \Sigma_n} \sum_{k=1}^{\#\VV_\jjj} \alpha_2(\jjj)^{-\tau(q)+\kappa} \mu([\jjj])^q \biggl(\frac{\alpha_2(\jjj)}{\alpha_1(\jjj)}\biggr)^q \\ 
    &\le cC^q(2\roo)^t \sum_{\jjj \in \Sigma_n} \alpha_2(\jjj)^\kappa \mu([\jjj])^q \fii^{s_q}(\jjj)^{1-q} \\ 
    &\le cC^q(2\roo)^t e^{n\gamma}.
  \end{align*}
  Consequently, under either assumption,
  \begin{align*}
    \sum_{B \in \VV} \diam(B)^t &= \sum_{n=1}^\infty \sum_{\jjj \in \Sigma_n} \sum_{B \in \VV_\jjj} \diam(B)^t \\ 
    &\le c(2\roo)^t(2C)^q \sum_{n=1}^\infty e^{n\gamma} = c(2\roo)^t(2C)^q \frac{e^{\gamma}}{1-e^{\gamma}}.
  \end{align*}
  Since the upper bound above does not depend on $\delta$, we have shown \eqref{eq:upper-bound-large-q-sum-finite} and finished the proof.
\end{proof}

\begin{prop} \label{thm:main-large-q}
  Let $X$ be a dominated planar self-affine set satisfying the strong separation condition such that $X_F$ is not a singleton, $\mu \in \MM(\Sigma)$ be a fully supported quasi-Bernoulli measure, and let $q > 0$ be such that $q \ne 1$. If $\tau(q)/(q-1)$ and $q\tau'(q)-\tau(q)$ are contained in $(0,1)$, and there exist constants $\roo_3>0$ and $C \ge 1$ such that
    \begin{equation} \label{eq:main-large-q1}
      \pi_*\mu(B(\pi(\iii),\roo_3\alpha_1(\iii|_n))) \le C\mu([\iii|_n])
    \end{equation}
    for all $\iii \in \Sigma$ and $n \in \N$, then
  \begin{equation*}
    \dimh(X_{\tau'(q)}) \ge q\tau'(q)-\tau(q).
  \end{equation*}
\end{prop}

\begin{proof}
  Let $\nu$ be the equilibrium state for $\psi^{q,\tau(q)/(q-1)}$. The containment $\fii_{\iii|_n}(X) \subset B(\pi(\iii),\diam(X)\alpha_1(\iii|_n))$ and \eqref{eq:main-large-q1} imply
  \begin{equation*}
    \frac{1}{C}\pi_*\mu(B(\pi(\iii),\roo_3\alpha_1(\iii|_n))) \le \mu([\iii|_n]) \le \pi_*\mu(B(\pi(\iii),\diam(X)\alpha_1(\iii|_n))).
  \end{equation*}
  Since $\log C/\log \alpha_1(\iii|_n) \to 0$ for $\nu$-almost all $\iii$, Lemma \ref{thm:fekete} and Proposition \ref{thm:spectrum-derivative} yield $\dimloc(\pi_*\mu,\pi(\iii)) = \tau'(q)$ for $\nu$-almost all $\iii$. Hence $\pi_*\nu(X_{\tau'(q)})=1$. Since $\nu$ is quasi-Bernoulli, \eqref{eq:BHR} and Theorem \ref{thm:dim-tau-derivative} imply $\ldimh(\pi_*\nu) = q\tau'(q)-\tau(q)$. Therefore
  \begin{equation*}
    \dimh(X_{\tau'(q)}) \ge \ldimh(\pi_*\nu) = q\tau'(q)-\tau(q)
  \end{equation*}
  and the proof is finished.
\end{proof}

\begin{prop} \label{thm:main-small-q-large-other}
  Let $X$ be a dominated planar self-affine set satisfying the strong separation condition such that $X_F$ is not a singleton, $\mu \in \MM(\Sigma)$ be a fully supported quasi-Bernoulli measure, and $q > 0$ is such that $q \ne 1$. If $\tau(q)/(q-1)$ and $q\tau'(q)-\tau(q)$ are contained in $(1,2)$, and
  \begin{equation} \label{eq:main-small-q-large-other1}
    \dimloc((\proj_{V^\bot})_*(\pi_*\mu),\proj_{V^\bot}(\pi(\iii))) = 1
  \end{equation}
  for $\nu \times \nu_F$-almost all $(\iii,V) \in \Sigma \times \RP$, where $\nu \in \MM_\sigma(\Sigma)$ is the equilibrium state for $\psi^{q,\tau(q)/(q-1)}$, then
  \begin{equation*}
    \dimh(X_{\tau'(q)}) \ge q\tau'(q)-\tau(q).
  \end{equation*}
\end{prop}

\begin{proof}
  Set $s=\tau(q)/(q-1)$ and $t=q\tau'(q)-\tau(q)$. By assumption, $s,t \in (1,2)$ and therefore
  \begin{equation*}
    \tau'(q) = \frac{t+(q-1)s}{q} \in (1,2).
  \end{equation*}
  Since $\nu$ is quasi-Bernoulli, \eqref{eq:BHR} and Theorem \ref{thm:dim-tau-derivative} imply that
  \begin{equation*}
    \ldimh(\pi_*\nu) = \diml(\nu) = q\tau'(q)-\tau(q).
  \end{equation*}
  Proposition \ref{thm:spectrum-derivative} gives $\diml(\mu,\nu) = \tau'(q)$. Therefore, the assumption \eqref{eq:main-small-q-large-other1} and Theorem \ref{thm:exact-dim} imply that
  \begin{equation*}
    \dimloc(\pi_*\mu,\pi(\iii)) = \min\{2,\diml(\mu,\nu)\} = \tau'(q)
  \end{equation*}
  for $\nu$-almost all $\iii \in \Sigma$. Thus $\pi_*\nu(X_{\tau'(q)})=1$ and so $\dimh(X_{\tau'(q)}) \ge \ldimh(\pi_*\nu)$ as claimed.
\end{proof}

We are now ready to prove the main results stated in the introduction.

\begin{proof}[Proof of Theorem \ref{1.2}]
  By Proposition \ref{thm:upper-bound-small-q}, we have
  \begin{equation} \label{eq:main1-bound}
    \dimh(X_{\tau'(q)}) \le q\tau'(q)-\tau(q)
  \end{equation}
  whenever $0<q<1$. The projective strong separation condition \eqref{eq:PSSC} and the compactness of both the self-affine set $X$ and the set of Furstenberg directions $X_F$ guarantee that
  \begin{equation*}
    \roo = \min_{V \in X_F} \min_{i \ne j} \dist(\proj_{V^\bot}(\conv(\fii_i(X))), \proj_{V^\bot}(\conv(\fii_i(X)))) > 0.
  \end{equation*}
  Fix $V \in X_F$ and let us show that
  \begin{equation} \label{eq:main1-iteration}
    \dist(\proj_{V^\bot}(\conv(\fii_\iii(X))), \proj_{V^\bot}(\conv(\fii_\iii(X)))) \ge \|\proj_{V^\bot}A_\iii\| \roo
  \end{equation}
  for all $\iii, \jjj \in \Sigma_n$ with $\iii \ne \jjj$. Write $\hhh = \iii \wedge \jjj$ and note that $\iii|_{|\hhh|+1} = \hhh i$ and $\jjj|_{|\hhh|+1} = \hhh j$ where $i \ne j$. Since $(\proj_{V^\bot}A)^\top = A^\top\proj_{V^\bot}$ for all $A \in GL_2(\R)$, we have
  \begin{equation*}
    \proj_{V^\bot}(\fii_\hhh(x)) = \pm\|A_\hhh^\top|V^\bot\|\proj_{V^\bot}(x) + \proj_{V^\bot}(\fii_\hhh(0)),
  \end{equation*}
  where $\pm$ indicates that the equality holds for one of the signs. Define a mapping $g \colon V^\bot \to V^\bot$ by setting
  \begin{equation*}
    g(x) = \pm\|A_\hhh^\top|V^\bot\|x + \proj_{V^\bot}(\fii_\hhh(0))
  \end{equation*}
  for all $x \in V^\bot$. If \eqref{eq:main1-iteration} does not hold, then there are $x \in \proj_{V^\bot}(\conv(\fii_\iii(X)))$ and $y \in \proj_{V^\bot}(\conv(\fii_\jjj(X)))$ such that $|x-y| < \|\proj_{V^\bot}A_\iii\| \roo$. Since
  \begin{equation*}
    |g^{-1}(x)-g^{-1}(y)| = \|A_\hhh^\top|V^\bot\|^{-1}|x-y| < \roo,
  \end{equation*}
  where $g^{-1}(x) \in \proj_{V^\bot}(\conv(\fii_i(X)))$ and $g^{-1}(y) \in \proj_{V^\bot}(\conv(\fii_j(X)))$, we have a contradiction with the choice of $\roo$. Therefore \eqref{eq:main1-iteration} holds.

  By domination, it follows from \cite[Lemma 2.8]{BaranyKaenmakiYu2021} that there exists a constant $D \ge 1$ such that
  \begin{equation*}
    \|\proj_{V^\bot}A_\iii\| \le \alpha_1(\iii|_n) \le D\|\proj_{V^\bot}A_\iii\|
  \end{equation*}
  for all $\iii \in \Sigma_*$ and $V \in X_F$. Hence, by \eqref{eq:main1-iteration}, we see that
  \begin{equation*}
    B(\pi(\iii),\roo (2D)^{-1}\alpha_1(\iii|_n)) \cap \fii_\jjj(X) = \emptyset
  \end{equation*}
  for all $\jjj \in \Sigma_n \setminus \{\iii|_n\}$ and, in particular,
  \begin{equation*}
    \pi_*\mu(B(\pi(\iii),\roo (2D)^{-1}\alpha_1(\iii|_n))) \le \mu([\iii|_n])
  \end{equation*}
  for all $\iii \in \Sigma$, $n \in \N$, and for any measure $\mu \in \MM(\Sigma)$. Therefore, Proposition \ref{thm:upper-bound-large-q}(1) extends \eqref{eq:main1-bound} for all $0<q\ne 1$ and Proposition \ref{thm:main-large-q} gives the lower bound.
\end{proof}

\begin{proof}[Proof of Theorem \ref{1.3}]
  Let $\mu \in \MM(\Sigma)$ be a fully supported quasi-Bernoulli measure. By Proposition \ref{thm:upper-bound-small-q}, we have
  \begin{equation} \label{eq:main2-bound}
    \dimh(X_{\tau'(q)}) \le q\tau'(q)-\tau(q)
  \end{equation}
  whenever $0<q<1$. Since for each $V \in X_F$ there exists a Lebesgue integrable function $g_V \colon V^\bot \to \R$ such that $C=\sup_{V \in X_F}\|g_V\|_\infty < \infty$ and
  \begin{equation*}
    (\proj_{V^\bot})_*(\pi_*\mu)(A) = \int_A g_V \dd\LL^1
  \end{equation*}
  for all Borel sets $A \subset V^\bot$, where $\LL^1$ is the Lebesgue measure on $V^\bot$, we have
  \begin{equation*}
    (\proj_{V^\bot})_*(\pi_*\mu)(\proj_{V^\bot}(B(x,r))) \le 2Cr
  \end{equation*}
  for all $x \in \R^2$, $V \in X_F$, and $r>0$. Therefore, Proposition \ref{thm:upper-bound-large-q}(2) extends \eqref{eq:main2-bound} for all $0<q\ne 1$ and Proposition \ref{thm:main-small-q-large-other} gives the lower bound.
\end{proof}

To finish the discussion, let us prove Proposition \ref{thm:abs-cont}. The Fourier transform of a finite Borel measure $\mu$ on $\R^2$ is defined by
\begin{equation*}
  \hat\mu(\xi) = \int_{\R^2} e^{i\langle \xi,x \rangle} \dd\mu(x)
\end{equation*}
for all $\xi \in \R^2$ and the Fourier transform of an integrable function $\psi \colon \R^d \to \R$ by
\begin{equation*}
  \hat\psi(\xi) = \int_{\R^d} \psi(\vvv) e^{i\langle \xi,\vvv \rangle} \dd\vvv
\end{equation*}
for all $\xi \in \R^d$, where $\langle \,\cdot\,,\cdot\, \rangle$ denotes the standard inner product on $\R^d$. If $\psi$ is a smooth function with compact support on $\R^d$, then it follows from the Riemann-Lebesgue lemma and integration by parts that for every $M$ there is a constant $C_M>0$ such that
\begin{equation} \label{eq:riemann-lebesgue}
  |\hat\psi(\xi)|\leq \frac{C_M}{(1+|\xi|)^M}
\end{equation}
for all $\xi\in\R^d$. There are intimate connections between the regularity of a measure and the decay rate of its Fourier transform. In particular, if $\mu$ is a finite Borel measure with compact support on $\R^2$ such that
\begin{equation} \label{eq:abs-cont-via-fourier}
  \int_{\R^2} |\xi|^\beta |\hat\mu(\xi)|^2 \dd\xi < \infty
\end{equation}
for some $\beta>2$, then, by \cite[Theorem 5.4]{Mattila2015}, $\mu$ is absolutely continuous with bounded continuous density.

\begin{proof}[Proof of Proposition \ref{thm:abs-cont}]
  Assuming that $\mu$ is a fully supported self-affine measure satisfying
  \begin{equation} \label{eq:abs-cont-mu3}
    \sum_{i=1}^N |c_i|^{-s}\mu(\fii_i(X))^2 < 1
  \end{equation}
  for some $s>3$, the proof requires only minor adjustments to that of \cite[Proposition 4.2]{Barany2025}: we replace the equilibrium state with $\mu$ and invoke assumption \eqref{eq:abs-cont-mu3} at the conclusion to ensure the finiteness of the integral. The key steps are outlined below. We indicate at the end of the proof how the extension to quasi-Bernoulli measures is achieved.

  We write $\mu_\vvv$ to emphasize the dependence of the self-affine measure on $\vvv = (v_1^2,\ldots,v_N^2)$. Let $s>3$ be such that \eqref{eq:abs-cont-mu3} holds and choose $2<\beta<s-1$. Choose $h \colon \RP \to [0,\infty)$ to be a compactly supported continuous density function uniformly separated away from zero on the smallest projective interval $[-C,C]$ containing $X_F$. We identify $\RP$ with $[0,1)$. Recalling \eqref{eq:riemann-lebesgue}, let $C_1>0$ be such that
  \begin{equation*}
    |\hat h(\xi)|\leq \frac{C_1}{(1+|\xi|)^s}
  \end{equation*}
  for all $\xi \in \R$. We define a compactly supported measure $\nu_\vvv$ on $\R^2$ by
  \begin{equation*}
    \nu_\vvv(A) = \iint_A h(x)\dd(\proj_x)_*\mu_\vvv(y)\dd x
  \end{equation*}
  for all Borel sets $A \subset \R^2$, where $\proj_x$ with $\mathrm{im}\proj_x = \R$ corresponds to the orthogonal projection $\proj_{V^\perp}$ via the identification. Let $\psi \colon \R^N \to [0,\infty)$ be a compactly supported density function and, by \eqref{eq:riemann-lebesgue}, let $C_2>0$ be such that
  \begin{equation*}
    |\hat\psi(\xi)|\leq \frac{C_2}{(1+\|\xi\|)^s}
  \end{equation*}
  for all $\xi \in \R^N$. By the proof of \cite[Proposition 4.2]{Barany2025} and \eqref{eq:abs-cont-mu3}, there is a constant $C_3>0$ such that
  \begin{equation} \label{eq:abs-cont-mu3-proof}
  \begin{split}
    \biggl| \int_{\R^N} &\int_{\R^2} |\xi|^\beta |\hat\nu_\vvv(\xi)|^2 \dd\xi \;\psi(\vvv)\dd\vvv \biggr| \le \cdots \\ 
    &\le C_3\int_{\R^2} \frac{|(\xi_1,\xi_2)|^\beta}{(1+|\xi_1|)^{s}(1+|\xi_2|)^{s}} \dd\xi_1\dd\xi_2 \sum_{n \in \N} \biggl( \sum_{i=1}^N \frac{\mu(\fii_i(X))^2}{|c_i|^s} \biggr)^n \\ 
    &< \infty,
  \end{split}
  \end{equation}
  where the calculation omitted by the ellipsis is identical to that in the proof of \cite[Proposition 4.2]{Barany2025}. Therefore, for Lebesgue almost every choice of $\vvv$, by \eqref{eq:abs-cont-via-fourier}, $\nu_\vvv$ is absolutely continuous with bounded continuous density, i.e., there exists a continuous $g_\vvv \colon \R^2 \to [0,\infty)$ such that $\nu_\vvv(A) = \int_A g_\vvv(x,y)\dd(x,y)$ for all Borel sets $A \subset \R^2$. Since $h$ is uniformly separated away from zero, it follows that $(\proj_x)_*\mu$ is absolutely continuous with uniformly bounded continuous density $g_\vvv(x,y)/h(x)$.

  The proof extends to quasi-Bernoulli measures $\mu$ by invoking the assumption \eqref{eq:abs-cont-qb} at the calculation \eqref{eq:abs-cont-mu3-proof} to ensure the finiteness of the integral in question.
\end{proof}

\begin{ack}
  In retrospect, AK gratefully acknowledges discussions with Thomas Jordan, Pablo Shmerkin, and K\'aroly Simon around 2012, which served as a motivating catalyst and significantly shaped the trajectory of this work. Additionally, he expresses gratitude to Bal\'azs B\'ar\'any for valuable conversations related to Proposition \ref{thm:abs-cont}. 
\end{ack}

{\bf Declaration.} The authors have no relevant financial or non-financial interests to disclose. This work was not supported by grants from any funding bodies.

\bibliographystyle{abbrv} 
\bibliography{SA}

\addresses

\end{document}